\documentclass{article}

\usepackage{color}

\usepackage{latexsym}
\usepackage{amssymb}
\usepackage{amsthm}
\usepackage{amssymb,amsmath}
\usepackage{MnSymbol}
\usepackage{mathtools}
\usepackage{stmaryrd}
\usepackage{skak}
\usepackage[all]{xy}	
\usepackage{mathrsfs}
\usepackage{cancel}
\usepackage{harpoon}
\usepackage{yfonts}

\DeclareMathAlphabet{\mathpzc}{OT1}{pzc}{m}{it}

\numberwithin{equation}{section}

\newtheorem{thm}{Theorem}[section]
\newtheorem{lem}{Lemma}[section]
\newtheorem{cor}{Corollary}[section]
\newtheorem{prop}{Proposition}[section]

\theoremstyle{definition}
\newtheorem{rem}{Remark}[section]

\title{Primitive bound of a 2-structure}

\author{Abderrahim Boussa\"{\i}ri\thanks{Facult\'e des Sciences A\"{\i}n Chock, 
D\'epartement de Math\'ematiques et Informatique, Km 8 route d'El Jadida, 
BP 5366 Maarif, Casablanca, Maroc; {\tt aboussairi@hotmail.com}.}
\and Pierre Ille\thanks{Aix Marseille Universit\'e, CNRS, Centrale Marseille, I2M, UMR 7373, 13453 Marseille, France; {\tt pierre.ille@univ-amu.fr}.}
\thanks{Centre de recherches math\'ematiques, Universit\'e de Montr\'eal, Case postale 6128, 
Succursale Centre-ville, Montr\'eal, Qu\'ebec, Canada H3C 3J7.}\and Robert E. Woodrow\thanks
   {Department of Mathematics and Statistics, The University of Calgary, 2500 University Drive, Calgary, Alberta, Canada T2N 1N4;
   {\tt woodrow@ucalgary.ca}.}}

\begin{document}

\maketitle

\begin{abstract}
A 2-structure on a set $S$ is given by an equivalence relation on the set of ordered pairs of distinct elements of $S$.  A subset $C$ of $S$, any two elements of which appear the same from the perspective of each element of the complement of $C$, is called a clan.  The number of elements that must be added in order to obtain a 2-structure the only clans of which are trivial is called the primitive bound of the 2-structure.  The primitive bound is determined for arbitrary 2-structures of any cardinality.  This generalizes the classical results of Erd\H{o}s et al. and Moon for tournaments, as well as the result of Brignall et al. for finite graphs, and the precise results of Boussa\"{\i}ri and Ille for finite graphs, providing new proofs which avoid extensive use of induction in the finite case.
\end{abstract}

\medskip


\medskip


\section{Introduction}\label{intro}

The notion of a primitive \cite{EHR99,F53} (also called indecomposable \cite{F84,I97}, prime \cite{CH93,G97} or simple \cite{B07,BRV11,EFHM72,EHM72,FL71,M72}) structure has been studied in the context of graphs, tournaments, more general structures 
derived from binary relations, and in general for relational structures by Fra\"{\i}ss\'e \cite{F53,F84}. 
Key is the idea of a subset the elements of which look the same from the perspective of each element of the complement, called autonomous set 
\cite{G67,K85,MP01}, clan \cite{EHR99}, convex set \cite{EFHM72,EHM72}, homogeneous set \cite{CH93}, interval \cite{F84,I97,ST93}, module \cite{BE65,G97,S92} or partitive set \cite{S73}. 
An indecomposable structure is one for which all such subsets are trivial. 

Given a structure, it is natural to ask about embedding it into an indecomposable structure and to seek to mimimize the number of elements one must add. 
In the early 70's this was done by Sumner \cite{S71} for finite complete graphs, by Moon~\cite{M72} for finite tournaments and by Erd\H{o}s et al.~\cite{EHM72} for arbitrary tournaments. 
More recently the question was revived by Brignall~\cite{B07} and Brignall et al. \cite{BRV11} for finite graphs and other finite combinatorial structures. 
Boussairi and Ille~\cite{BI10} provided a detailed analysis and identified the precise parameter to describe the situation for finite graphs. 
In this work we yield a unified approach by studying the more general situation of 2-structures \cite{EHR99}. 
The techniques introduced permit us to generalize the parameter obtained by 
Boussairi and Ille~\cite{BI10} to the arbitrary setting. 

Resolution of the problem for tournaments by Erd\H{o}s et al.~\cite{EFHM72} employed a linearization of the tournament and the Bernstein Property. 
We introduce the notion of a traverse of an arbitrary 2-structure which respects key clans of the 2-structure and the notion of a dense bicoloration mirroring the Bernstein Property. 
We also introduce the notion of inclusive clans and develop their structural properties. 
These tools permit us to provide precise bounds and to present a new proof of the now classical result of 
Erd\H{o}s et al.~\cite{EHM72} that adds structural understanding and does not turn on induction to handle the finite case. 

The notion of traverse plays an important role in the algorithmics for the finite graphs as well.  
There a traverse induces a permutation of the vertex set called a factorizing permutation \cite{C97}.  Factorizing permutations are used to find efficient algorithms which compute the clan tree. 
Given a finite graph, the first step consists in calculating a factorizing permutation \cite{HPV99} and the second in determining the clan tree from a factorizing permutation \cite{CHM02}. 

At present, we formalize our presentation. 
A {\em 2-structure} $\sigma$ consists of an infinite or finite {\em vertex set} (or {\em domain}~\cite{EHR99}) $V(\sigma)$ and of an equivalence relation $\equiv_\sigma$ defined on $(V(\sigma)\times V(\sigma))\setminus\{(v,v):v\in V(\sigma)\}$. 
Set $\nu(\sigma)=|V(\sigma)|$ where $|V(\sigma)|$ denotes the cardinality of $V(\sigma)$. 
The family of the equivalence classes of $\equiv_\sigma$ is denoted by 
$E(\sigma)$. 
Set $\varepsilon(\sigma)=|E(\sigma)|$. 
Given a 2-structure $\sigma$, 
with each $W\subseteq V(\sigma)$ associate the {\em 2-substructure} $\sigma[W]$ of $\sigma$ induced by $W$ defined on $V(\sigma[W])=W$ such that 
$$(\equiv_{\sigma[W]})\ =\ (\equiv_{\sigma})_{\restriction (W\times W)\setminus\{(w,w):w\in W\}}.$$
Given $W\subseteq V(\sigma)$, 
$\sigma[V(\sigma)\setminus W]$ is denoted by $\sigma-W$ and by $\sigma-w$ when $W=\{w\}$.

With each 2-structure $\sigma$ associate the 2-structure $\sigma^{\star}$ defined on $V(\sigma^{\star})=V(\sigma)$ by 
$$(u,v)\equiv_{\sigma^{\star}}(x,y)\quad\text{if}\quad
(v,u)\equiv_{\sigma}(y,x)$$
for any $u\neq v\in V(\sigma^{\star})$ and $x\neq y\in V(\sigma^{\star})$. 
A 2-structure $\sigma$ is {\em reversible} if $\sigma=\sigma^{\star}$. 
Let $\sigma$ be a reversible 2-structure. 
For each $e\in E(\sigma)$, $e^{\star}=\{(u,v):(v,u)\in e\}\in E(\sigma)$ and we have either $e=e^{\star}$ or $e\cap e^{\star}=\emptyset$. 
In the first instance, $e$ is said to be {\em symmetric}. 
It is called {\em asymmetric} in the second. 
The family of the asymmetric classes of $\equiv_\sigma$ is denoted by 
$E_a(\sigma)$ and that of the symmetric ones by $E_s(\sigma)$. 
Set $\varepsilon_a(\sigma)=|E_a(\sigma)|$ and 
$\varepsilon_s(\sigma)=|E_s(\sigma)|$. 
A reversible 2-structure $\sigma$ is {\em symmetric} when $E(\sigma)=E_s(\sigma)$, it is {\em asymmetric} when 
$E(\sigma)=E_a(\sigma)$. 

A graph $\Gamma=(V(\Gamma),E(\Gamma))$ is identified with the symmetric 2-structure $\sigma(\Gamma)$ defined on 
$V(\sigma(\Gamma))=V(\Gamma)$ as follows. 
For any $u\neq v\in V(\Gamma)$ and $x\neq y\in V(\Gamma)$, 
$(u,v)\equiv_{\sigma(\Gamma)}(x,y)$ if either 
$\{u,v\},\{x,y\}\in E(\Gamma)$ or $\{u,v\},\{x,y\}\not\in E(\Gamma)$. 
Notice that 
$\sigma(\Gamma)=\sigma(\overline{\Gamma})$ where 
$\overline{\Gamma}=(V(\Gamma),\binom{V(\Gamma)}{2}\setminus E(\Gamma))$ is the {\em complement} of $\Gamma$. 
%Given $v\in V(\Gamma)$, 
%$N_\Gamma(v)$ denotes the {\em neighbourhood} of $v$ in $\Gamma$, and 
%$d_\Gamma(v)$ denotes its {\em degree}. 
Similarly, 
a tournament $T=(V(T),A(T))$ is identified with the asymmetric 2-structure $\sigma(T)$ defined on 
$V(\sigma(T))=V(T)$ as follows. 
For any $u\neq v\in V(T)$ and $x\neq y\in V(T)$, 
$(u,v)\equiv_{\sigma(T)}(x,y)$ if either $(u,v),(x,y)\in A(T)$ or $(u,v),(x,y)\not\in A(T)$. 
Notice that 
$\sigma(T)=\sigma(T^{\star})$ where 
$T^{\star}=(V(T),A(T)^{\star})$ is the {\em dual} of $T$.

Let $\sigma$ be a 2-structure. 
A subset $C$ of $V(\sigma)$ is a {\em clan} \cite{EHR99} of $\sigma$ if for any $c,d\in C$ and 
$v\in V(\sigma)\setminus C$, we have 
$$(c,v)\equiv_{\sigma}(d,v)\ \text{and}\ (v,c)\equiv_{\tau}(v,d).$$
For instance, 
$\emptyset$, $V(\sigma)$ and $\{v\}$, for $v\in V(\sigma)$, are clans of $\sigma$ called {\em trivial} clans of $\sigma$. 
A 2-structure $\sigma$ is {\em primitive} \cite{EHR99} if $\nu(\sigma)\geq 3$ and if all of the clans of $\sigma$ are trivial. Otherwise $\sigma$ is said to be {\em imprimitive} \cite{EHR99}. 

Given a set $S$ with $|S|\geq 2$, 
Sumner~\cite[Theorem~2.45]{S71} observed that 
the complete graph $K_S$ admits a primitive graph extension $G$ such that 
$|V(G)\setminus S|=\lceil\log_2(|S|+1)\rceil$. 
This is extended to any graph in \cite[Theorem 3.7]{B07} and \cite[Theorem 3.2]{BRV11} as follows. 
A graph $G$, with $|V(G)|\geq 2$, 
admits a primitive graph extension $H$ such that 
$|V(H)\setminus V(G)|=\lceil\log_2(|V(G)|+1)\rceil$. 

A 2-structure $\tau$ is an {\em extension} of a 2-structure $\sigma$ if 
$V(\tau)\supseteq V(\sigma)$ and $\tau[V(\sigma)]=\sigma$. 
Given a cardinal $\kappa$, a $\kappa$-extension of a 2-structure $\sigma$ is an extension $\tau$ of $\sigma$ such that $|V(\tau)\setminus V(\sigma)|=\kappa$. 
Let $\tau$ be an extension of $\sigma$. 
Consider the function $\sigma\hookrightarrow\tau:E(\sigma)\longrightarrow E(\tau)$ satisfying 
$(\sigma\hookrightarrow\tau)(e)\supseteq e$ for every $e\in E(\sigma)$. 
Clearly $\sigma\hookrightarrow\tau$ is injective and 
we can identify $(\sigma\hookrightarrow\tau)(e)$ with $e$ for every 
$e\in E(\sigma)$. 
%So $(\sigma\hookrightarrow\tau)(e)$ is simply denoted by $e$ for every 
%$e\in E(\sigma)$. 
We say that $\tau$ is a {\em faithful} extension of $\sigma$ if 
\begin{equation}\label{E1faith}
\text{$\sigma\hookrightarrow\tau$ is bijective}
\end{equation}
and for any $e,f\in E(\sigma)$, 
\begin{equation}\label{E2faith}
(\sigma\hookrightarrow\tau)(e)\cap ((\sigma\hookrightarrow\tau)(f))^\star\neq\emptyset\ \Longrightarrow\ e\cap f^\star\neq\emptyset.
\end{equation}
The necessity of Conditions~\eqref{E1faith} and \eqref{E2faith} is discussed in Section~\ref{per2s}. 
In Sumner's result, the primitive graph extension is not a faithful extension of the complete graph. 
A 2-structure $\sigma$ is {\em complete} \cite{EHR99} if either $\nu(\sigma)\leq 1$ or $\nu(\sigma)\geq 2$ and $\varepsilon(\sigma)=1$. 
A subset $W$ of $V(\sigma)$ is {\em complete} if $\sigma[W]$ is complete. 
Clearly complete subsets of 2-structures correspond to cliques and stables sets in graphs. As already observed in  Sumner's result, to obtain primitive extensions of a complete 2-structures, we consider symmetric 2-structures $\sigma$ with $\varepsilon(\sigma)= 2$. 

Let $\sigma$ be a 2-structure admitting a primitive and faithful extension. 
The {\em primitive bound} $p(\sigma)$ of $\sigma$ is the smallest cardinal $\kappa$ such that $\sigma$ possesses a primitive and faithful $\kappa$-extension. 
In these terms, 
Brignall et al. \cite{BRV11} obtained that $$p(G)\leq\lceil\log_2(|V(G)|+1)\rceil$$ for any finite graph $G$. 
In his Ph.D. Thesis, 
Brignall \cite{B07} conjectured that $$p(G)\leq\lceil\log_2(\max(\alpha(G),\omega(G))+1)\rceil.$$
Boussa\"{\i}ri and Ille \cite{BI10} identified the correct parameter 
$$c(G)=\max(\{|C|:C\ \text{is a clan of $G$ which is a stable set or a clique\})}$$ and proved for every finite graph $G$ that 
$$\lceil\log_2(c(G))\rceil\leq p(G)\leq\lceil\log_2(c(G)+1)\rceil.$$
Thus $p(G)=\lceil\log_2(c(G))\rceil$ when $\log_2(c(G))\not\in\mathbb{N}$. 
When $c(G)=2^k$ with $k\geq 1$, $p(G)=k$ or $k+1$ and Boussa\"{\i}ri and Ille \cite{BI10} showed that $p(G)=k+1$ if and only if $G$ (or $\overline{G}$) admits $2^k$ isolated vertices. 

For a 2-structure $\sigma$, we introduce the {\em clan completness} $c(\sigma)$ of $\sigma$ as being the supremum of cardinalities of complete clans of $\sigma$. 
Furthermore consider a reversible 2-structure $\sigma$. 
Given $e\in E_s(\sigma)$, 
a vertex $v$ of $\sigma$ is {\em $e$-isolated} if 
$(v,w)\in e$ for every $w\in V(\sigma)\setminus\{w\}$. 
The family of $e$-isolated vertices is denoted by $\odot_e(\sigma)$. 

Given a reversible 2-structure $\sigma$ such that $2\leq\varepsilon(\sigma)\leq c(\sigma)<\aleph_0$, we prove in Corollary~\ref{C1bound} that 
$$\lceil\log_{\varepsilon(\sigma)}(c(\sigma))\rceil\leq p(\sigma)\leq\lceil\log_{\varepsilon(\sigma)}(c(\sigma)+1)\rceil.$$ 
Moreover, when $c(\sigma)=\varepsilon(\sigma)^{k}$ where $k\geq 1$, we show in Theorem~\ref{T3Abound} that 
$$p(\sigma)=k+1\ \iff\ \text{there is $e\in E_s(\sigma)$ such that $|\odot_e(\sigma)|=\varepsilon(\sigma)^{k}$.}$$

The cardinal logarithm is defined as follows. 
Given cardinals $\mu$ and $\nu$, $$\mathfrak{log}_\mu(\nu)=\min(\{\kappa:\mu^\kappa\geq\nu\}).$$ 
If $\mu$ and $\nu$ are finite, then $\mathfrak{log}_\mu(\nu)=\lceil \log_\mu(\nu)\rceil$. 
Given a reversible 2-structure $\sigma$ such that $\varepsilon(\sigma)\geq 2$, we establish in 
Theorem~\ref{T1bound} that 
$$\mathfrak{log}_{\varepsilon(\sigma)}(c(\sigma))\geq\aleph_0\ \Longrightarrow\ 
p(\sigma)=\mathfrak{log}_{\varepsilon(\sigma)}(c(\sigma)).$$

For an infinite or finite tournament $T$, Erd\H{o}s et al. \cite{EFHM72} established that $p(T)\leq 2$. 
Then Moon \cite{M72} proved for a finite tournament $T$ such that $|V(T)|\geq 4$ that 
$p(T)=2$ if and only if $T$ is an odd linear order. 
Erd\H{o}s et al. \cite{EHM72} extended this result to arbitrary tournament. 
Using dense bicolorations of traverses (see Section~\ref{sectiontrav}) and inclusive clans 
(see Section~\ref{sectioninclusive}), we provide an elegant proof of \cite{EHM72} 
(see Theorem~\ref{T4bound}). 

We determine the primitive bounds of the other reversible 2-structures in Theorem~\ref{T2bound}. 
In particular, we obtain the following for a reversible 2-structure $\sigma$ such that $c(\sigma)=1$. 
If $\varepsilon(\sigma)\geq 3$ or if $\varepsilon(\sigma)=\varepsilon_s(\sigma)=2$, then 
$p(\sigma)\leq 1$. 
For non reversible 2-structures, we proceed as follows. 
With 2-structures $\sigma$ and $\tau$ such that $V(\sigma)=V(\tau)$ associate the 2-structure 
$\sigma\land\tau$ defined on $V(\sigma\land\tau)=V(\sigma)$ by 
$$E(\sigma\land\tau)=\{e\cap f:e\in E(\sigma),f\in E(\tau)\ \text{and}\ e\cap f\neq\emptyset\}.$$
For an arbitrary 2-structure $\sigma$, $\sigma\land\sigma^{\star}$ is reversible and we prove in 
Theorem~\ref{T1nonrev} that $p(\sigma)=p(\sigma\land\sigma^\star)$.

\section{Clan tree}

Given a 2-structure $\sigma$, we use the following notation. 
For any $u\neq v\in V(\sigma)$, the equivalence class of $\equiv_\sigma$ containing $(u,v)$ is denoted by $(u,v)_\sigma$. 
We define the function $\overrightarrow{\sigma}$ on $V(\sigma)$ as follows. 
For each $v\in V(\sigma)$, 
$\overrightarrow{\sigma}(v):V(\sigma)\setminus\{v\}\longrightarrow E(\sigma)$ with 
$w\longmapsto (v,w)_\sigma$ for every $w\in V(\sigma)\setminus\{v\}$. 
With $e\in E(\sigma)$ 
associate the function $\bar{e}:V(\sigma)\longrightarrow E(\sigma)$ defined by $v\longmapsto e$ for every $v\in V(\sigma)$. 

For $W,W'\subseteq V(\sigma)$, with $W\cap W'=\emptyset$, $W\longleftrightarrow_\sigma W'$ signifies that $(v,v')\equiv_\sigma(w,w')$ and $(v',v)\equiv_\sigma(w',w)$ for any 
$v,w\in W$ and $v',w'\in W'$. 
Given $v\in V(\sigma)$ and $W\subseteq V(\sigma)\setminus\{v\}$, 
$\{v\}\longleftrightarrow_\sigma W$ is also denoted by 
$v\longleftrightarrow_\sigma W$. 
The negation is denoted by $v\not\longleftrightarrow_\sigma W$. 
Let $W,W'\subseteq V(\sigma)$ such that $W\cap W'=\emptyset$ and 
$W\longleftrightarrow_\sigma W'$. 
The equivalence class of $(w,w')$, where $w\in W$ and $w'\in W'$, is denoted by 
$(W,W')_\sigma$. 
Given $W\subsetneq V(\sigma)$ and $v\in V(\sigma)\setminus W$ such that 
$v\longleftrightarrow_\sigma W$, 
$(\{v\},W)_\sigma$ is also denoted by $(v,W)_\sigma$. 

The family of the clans of a 2-structure $\sigma$ is denoted by $\mathbb{C}(\sigma)$. 
Furthermore set 
$\mathbb{C}_{\geq 2}(\sigma)=\{C\in\mathbb{C}(\sigma):|C|\geq 2\}$. 

Given a 2-structure $\sigma$, a partition $\mathbb{F}$ of $V(\sigma)$ is a {\em factorization} \cite{EHR99} of $\sigma$ if $\mathbb{F}\subseteq\mathbb{C}(\sigma)$. 
Let $\mathbb{F}$ be a factorization of $\sigma$.
Given $X,Y\in\mathbb{F}$, 
we have $X\longleftrightarrow_\sigma Y$ because $X\cap Y=\emptyset$. 
Thus there is $e\in E(\sigma)$ such that $(X,Y)_\sigma=e$. 
This justifies the following definition. 
The {\em quotient} of $\sigma$ by $\mathbb{F}$ is the 2-structure $\sigma/\mathbb{F}$ defined on $V(\sigma/\mathbb{F})=\mathbb{F}$ as follows. 
For any $X\neq Y\in\mathbb{F}$ and $X'\neq Y'\in\mathbb{F}$, 
$$(X,Y)\equiv_{\sigma/\mathbb{F}}(X',Y')\ \text{if}\ (X,Y)_\sigma=(X',Y')_\sigma.$$ 

The following strengthening of the notion of clan is useful to present the clan decomposition theorem. 
Given a 2-structure $\sigma$, a clan $C$ of $\sigma$ is said to be {\em prime} \cite{EHR99} provided that for every clan $D$ of $\sigma$, we have: 
$$\text{if $C\cap D\neq\emptyset$, then $C\subseteq D$ or $D\subseteq C$.}$$
%Notice that the trivial clans of $\sigma$ are prime. 
The family of prime clans of $\sigma$ is denoted by $\mathbb{P}(\sigma)$. 
Furthermore set $\mathbb{P}_{\geq 2}(\sigma)=\{C\in\mathbb{P}(\sigma):|C|\geq 2\}$. 
%\begin{rem}\label{R1prime}
%Given a 2-structure $\sigma$, 
%if $\sigma$ is complete or linear, then all of the subsets of $V(\sigma)$ are clans of $\sigma$. 
%It follows that 
%all of the prime clans of $\sigma$ are trivial. 
%\end{rem}

We associate with a 2-structure $\sigma$ the Gallai family $\mathbb{G}(\sigma)$ of the maximal elements under inclusion of $\mathbb{P}(\sigma)\setminus\{\emptyset,V(\sigma)\}$. 
Set 
$\mathbb{G}_1(\sigma)=\{X\in\mathbb{G}(\sigma):|X|=1\}$ and 
$\mathbb{G}_{\geq 2}(\sigma)=\mathbb{G}(\sigma)\setminus\mathbb{G}_1(\sigma)$. 
The clan decomposition theorem is stated as follows. It is attributable to Gallai \cite{G67,MP01} for finite graphs (see \cite[Theorem 5.5]{EHR99} for finite 2-structures and \cite[Theorem 4.2]{HR94} for infinite ones). 
Recall that an asymmetric 2-structure $\sigma$, with $\nu(\sigma)\geq 2$, is {\em linear} \cite{EHR99} 
if there is $e\in E(\sigma)$ such that $(V(\sigma),e)$ is a linear order. 

\begin{thm}\label{Tgallai}
For a 2-structure $\sigma$ such that $\mathbb{G}(\sigma)\neq\emptyset$, the family 
$\mathbb{G}(\sigma)$ realizes a factorization of $\sigma$. 
Moreover, the corresponding quotient $\sigma/\mathbb{G}(\sigma)$ is complete, linear or primitive. 
\end{thm}

Let $\sigma$ be a 2-structure. 
Given $C\in\mathbb{P}(\sigma)$, we have $\mathbb{G}(\sigma[C])\subseteq\mathbb{P}(\sigma)$. 
An element $C$ of $\mathbb{P}(\sigma)$ is a {\em limit} of $\sigma$ if 
$\mathbb{G}(\sigma[C])=\emptyset$. 
The family of the limits of $\sigma$ is denoted by $\mathbb{L}(\sigma)$. 
Now consider 
\begin{equation*}
\mathbb{T}(\sigma)=\{\{v\}:v\in V(\sigma)\}\cup (\bigcup_{C\in\mathbb{P}(\sigma)\setminus\mathbb{L}(\sigma)}
\mathbb{G}(\sigma[C])\cup\{C\}).
\end{equation*}
As a direct consequence of the definition of a prime clan, we obtain that 
the family $\mathbb{T}(\sigma)$ endowed with inclusion, denoted by 
$(\mathbb{T}(\sigma),\subseteq)$, is a tree called the {\em clan tree} of $\sigma$. 
For clan trees of finite digraphs, see \cite{CM05}. 
For infinite 2-structures or more generally for weakly partitive families on infinite sets, see \cite{IW09}. 
When $V(\sigma)$ is finite, we have 
$$\mathbb{T}(\sigma)=\bigcup_{C\in\mathbb{P}_{\geq 2}(\sigma)}
\mathbb{G}(\sigma[C])\cup\{C\}=\mathbb{P}(\sigma)\setminus\{\emptyset\}.$$

Let $\sigma$ be a 2-structure. 
Given $\emptyset\neq W\subseteq V(\sigma)$, $\bigcap\{C\in\mathbb{P}(\sigma):C\supseteq W\}$ is a prime clan of $\sigma$. 
It is denoted by $\widetilde{W}$. 
Similarly, given $\emptyset\neq W\subsetneq V(\sigma)$, $\bigcap\{C\in\mathbb{P}(\sigma):C\supsetneq W\}$ is a prime clan of $\sigma$ denoted by $\widehat{W}$. 
Since the proofs of the next three lemmas are easy, we omit them. 

\begin{lem}\label{L1tilde}
Let $\sigma$ be a 2-structure. 
For any $v\neq w\in V(\sigma)$, $\widetilde{\{v,w\}}\in\mathbb{P}(\sigma)\setminus\mathbb{L}(\sigma)$ and there are 
$X_v\neq X_w\in\mathbb{G}(\sigma[\widetilde{\{v,w\}}])$ such that $v\in X_v$ and $w\in X_w$.
\end{lem}

\begin{lem}\label{L1nonprime}
Let $\sigma$ be a 2-structure. For $C\subseteq V(\sigma)$, 
$C\in\mathbb{C}(\sigma)\setminus\mathbb{P}(\sigma)$ if and only if 
$\widetilde{C}\not\in\mathbb{L}(\sigma)$ and 
there is a nontrivial clan of $\sigma[\widetilde{C}]/\mathbb{G}(\sigma[\widetilde{C}])$ such that 
$C=\bigcup F$. 
\end{lem}

\begin{lem}\label{L1hat}
Given a 2-structure $\sigma$, consider $C\in\mathbb{P}(\sigma)\setminus\{\emptyset\}$. 
\begin{enumerate}
\item $C=\widehat{C}$ if and only if for each $X\in\mathbb{P}(\sigma)$, with 
$X\supsetneq C$, and for every $x\in X\setminus C$, there is $Z\in\mathbb{P}(\sigma)\setminus\mathbb{L}(\sigma)$ such that 
$C\subsetneq Z\subseteq X\setminus\{x\}$.
\item $C\subsetneq\widehat{C}$ if and only if $\widehat{C}\not\in\mathbb{L}(\sigma)$ and 
$C\in\mathbb{G}(\sigma[\widehat{C}])$.
\end{enumerate}
\end{lem}

Let $\sigma$ be a 2-structure. 
Using the Axiom of Choice and Theorem~\ref{Tgallai}, we label $\mathbb{P}(\sigma)\setminus\mathbb{L}(\sigma)$ by a function 
$\lambda_\sigma:\mathbb{P}(\sigma)\setminus\mathbb{L}(\sigma)\longrightarrow E(\sigma)\cup\{\varpi\}$ satisfying: 

\noindent for each $X\in\mathbb{P}(\sigma)\setminus\mathbb{L}(\sigma)$, 
\begin{itemize}
\item if $\sigma[X]/\mathbb{G}(\sigma[X])$ is primitive, then $\lambda_\sigma(X)=\varpi$;
\item if $\sigma[X]/\mathbb{G}(\sigma[X])$ is complete, then $\lambda_\sigma(X)=(Y,Z)_\sigma$ where $Y\neq Z\in\mathbb{G}(\sigma[X])$;
\item if $\sigma[X]/\mathbb{G}(\sigma[X])$ is linear, then we choose $(Y,Z)_\sigma$ or $(Z,Y)_\sigma$ for $\lambda_\sigma(X)$, where 
$Y\neq Z\in\mathbb{G}(\sigma[X])$.
\end{itemize}

With each $X\in\mathbb{P}(\sigma)\setminus\mathbb{L}(\sigma)$ such that $\sigma[X]/\mathbb{G}(\sigma[X])$ is linear, 
associate the linear order $O_X$ defined on $\mathbb{G}(\sigma[X])$ as follows. 
For any $Y\neq Z\in\mathbb{G}(\sigma[X])$, 
\begin{equation}\label{E1linear}
Y<Z\mod O_X\ \text{if}\ \ (Y,Z)_\sigma=\lambda_\sigma(X).
\end{equation}
Observe that the clans of $\sigma[X]/\mathbb{G}(\sigma[X])$ coincide with the intervals of $O_X$. 
A clan $H$ of $\sigma[X]/\mathbb{G}(\sigma[X])$ is said to be an {\em interval of singletons} of $O_X$ if $H\subseteq\mathbb{G}_1(\sigma[X])$. 
A {\em maximal interval of singletons} of $O_X$ is an interval of singletons of $O_X$ which is maximal among the intervals of singletons of $O_X$. 
Furthermore, given $Y,Z\in\mathbb{G}(\sigma[X])$, set 
$${\rm int}_X(Y,Z)=\bigcap\{F\in\mathbb{C}(\sigma[X])/\mathbb{G}(\sigma[X])):Y,Z\in F\}.$$

%Lastly, consider a 2-structure $\sigma$. 
%Given $X\in\mathbb{P}(\sigma)\setminus\mathbb{L}(\sigma)$, 
%notice that 
%\begin{equation*}
%\begin{cases}
%\lambda_\sigma(X)\in E_s(\sigma)\ \text{if $\sigma[X]/\mathbb{G}(\sigma[X])$ is complete},\\
%\text{and}\\
%\lambda_\sigma(X)\in E_a(\sigma)\ \text{if $\sigma[X]/\mathbb{G}(\sigma[X])$ is linear}.
%\end{cases}
%\end{equation*}

\section{Clan completness and tree equivalence}

In the section, we omit the proofs because they are somewhat technical, sometimes long and they do not present  a major interest in our topics. 

Let $\sigma$ be a 2-structure. 
Given $e\in E(\sigma)$, a subset $W$ of $V(\sigma)$ is {\em $e$-complete} if $(v,w)_\sigma=e$ for any $v\neq w\in W$. 
A subset $W$ of $V(\sigma)$ is {\em linear} if $\sigma[W]$ is linear. 
Given $e\in E(\sigma)$, a subset $W$ of $V(\sigma)$ is {\em $e$-linear} if 
$(W,\{(v,w)\in (W\times W)\setminus\{(u,u):u\in W\}):(v,w)_\sigma=e\}$ is a linear order. 
We consider the following subsets of $\mathbb{C}_{\geq 2}(\sigma)$: 
\begin{itemize}
\item $\mathscr{C}(\sigma)$ denotes the family of the maximal elements under inclusion of 
$\{C\in\mathbb{C}_{\geq 2}(\sigma):C\ \text{is complete}\}$;
\item $\mathscr{L}(\sigma)$ denotes the family of the maximal elements under inclusion of 
$\{C\in\mathbb{C}_{\geq 2}(\sigma):C\ \text{is linear}\}$;
\item $\mathscr{P}(\sigma)$ denotes the family of $C\in\mathbb{C}(\sigma)$ such that $\sigma[C]$ is primitive.
\end{itemize}
When $\mathscr{C}(\sigma)\neq\emptyset$, the clan completness $c(\sigma)$ of $\sigma$ satisfies 
$$c(\sigma)=\sup(\{|C|:C\in\mathscr{C}(\sigma)\}).$$ 

\begin{lem}\label{L1clanmax}
Given a 2-structure $\sigma$, consider $C\in\mathbb{C}_{\geq 2}(\sigma)$. 
\begin{enumerate}
\item If $\sigma[C]$ is complete, then there is $C'\in\mathscr{C}(\sigma)$ such that $C'\supseteq C$. 
\item If $\sigma[C]$ is linear, then there is $L\in\mathscr{L}(\sigma)$ such that $L\supseteq C$.
\end{enumerate}
\end{lem}

We characterize the elements of 
$\mathscr{C}(\sigma)\cup\mathscr{L}(\sigma)\cup\mathscr{P}(\sigma)$ in terms of the labelled clan tree of $\sigma$. 

\begin{prop}\label{P1clanmax}
For a 2-structure $\sigma$, the following three equivalences hold. 
\begin{enumerate}
\item Given $C\subseteq V(\sigma)$, 
$C\in\mathscr{C}(\sigma)$ if and only if $|C|\geq 2$, 
$\widetilde{C}\in\mathbb{P}(\sigma)\setminus\mathbb{L}(\sigma)$, 
$\lambda_\sigma(\widetilde{C})\in E_s(\sigma)$ and 
$C=\{c\in\widetilde{C}:\{c\}\in\mathbb{G}(\sigma[\widetilde{C}])\}$.
\item Given $L\subseteq V(\sigma)$, 
$L\in\mathscr{L}(\sigma)$ if and only if $|L|\geq 2$, 
$\widetilde{L}\in\mathbb{P}(\sigma)\setminus\mathbb{L}(\sigma)$, 
$\lambda_\sigma(\widetilde{L})\in E_a(\sigma)$ and 
$\{\{l\}:l\in L\}$ is a maximal interval of singletons of $O_{\widetilde{L}}$. 
\item Given $P\subseteq V(\sigma)$, 
$P\in\mathscr{P}(\sigma)$ if and only if $P\in\mathbb{P}(\sigma)\setminus\mathbb{L}(\sigma)$, 
$\lambda_\sigma(P)=\varpi$ and $\mathbb{G}(\sigma[P])=\mathbb{G}_1(\sigma[P])$. 
\end{enumerate}
\end{prop}

Let $\sigma$ be a 2-structure. 
As we will show (see Theorem~\ref{T1equivalence}), the elements of 
$\mathscr{C}(\sigma)\cup\mathscr{L}(\sigma)\cup\mathscr{P}(\sigma)$ are the nontrivial equivalence classes of the following equivalence relation. 
Given $v\neq w\in V(\sigma)$, 
$v\simeq_\sigma w$ if $\widehat{\{v\}}=\widehat{\{w\}}$ and if 
there is $F\in\mathbb{C}(\sigma[\widehat{\{v\}}]/\mathbb{G}(\sigma[\widehat{\{v\}}]))$ such that 
$\{v\},\{w\}\in F$ and $F\subseteq\mathbb{G}_1(\sigma[\widehat{\{v\}}])$. 
To make this definition clearer, observe the following. 
Let $v\neq w\in V(\sigma)$ such that $\widehat{\{v\}}=\widehat{\{w\}}$. 
Since $\{v\}\neq\widehat{\{v\}}$, it follows from Lemma~\ref{L1hat} that 
$\widehat{\{v\}}\in\mathbb{P}(\sigma)\setminus\mathbb{L}(\sigma)$ and 
$\{v\}\in\mathbb{G}_1(\sigma[\widehat{\{v\}}])$. 
Similarly $\{w\}\in\mathbb{G}_1(\sigma[\widehat{\{v\}}])$. 

\begin{lem}\label{L1equivalence}
Let $\sigma$ be a 2-structure. 
For $v\neq w\in V(\sigma)$, $v\simeq_\sigma w$ if and only if the following three assertions hold
\begin{itemize}
\item $\widehat{\{v\}}=\widehat{\{w\}}$;
\item if $\lambda_\sigma(\widehat{\{v\}})=\varpi$, then 
$\mathbb{G}(\sigma[\widehat{\{v\}}])=\mathbb{G}_1(\sigma[\widehat{\{v\}}])$;
\item if $\lambda_\sigma(\widehat{\{v\}})\in E_a(\sigma)$, then 
${\rm int}_{\widehat{\{v\}}}(\{v\},\{w\})\subseteq\mathbb{G}_1(\sigma[\widehat{\{v\}}])$. 
\end{itemize}
\end{lem}

Given a 2-structure $\sigma$, denote by $\mathfrak{C}(\sigma)$ the family of the equivalence classes of $\simeq_\sigma$. 
Set $\mathfrak{C}_1(\sigma)=\{C\in\mathfrak{C}(\sigma):|C|=1\}$ and 
$\mathfrak{C}_{\geq 2}(\sigma)=\{C\in\mathfrak{C}(\sigma):|C|\geq 2\}$.

\begin{thm}\label{T1equivalence}
For a 2-structure $\sigma$, 
$\mathfrak{C}_{\geq 2}(\sigma)=\mathscr{C}(\sigma)\cup\mathscr{L}(\sigma)\cup\mathscr{P}(\sigma)$. 
\end{thm}

Given a 2-structure $\sigma$, set 
$\Upsilon(\sigma)=\bigcup\mathfrak{C}_1(\sigma)$ and consider 
$\Upsilon_\downarrow(\sigma)=\{v\in\Upsilon(\sigma):\widehat{\{v\}}=\{v\}\}$. 
There are reversible 2-structures $\sigma$ such that $V(\sigma)=\Upsilon_\downarrow(\sigma)$. 
For such 2-structures $\sigma$, it follows from Theorems~\ref{T2bound} and \ref{T4bound} that 
$p(\sigma)=1$. 
The proofs use properties of traverses and inclusive clans introduced in the next two sections. 

%\begin{lem}\label{l*2}
%Given a reversible 2-structure $\sigma$, consider $C\in\mathbb{C}(\sigma)$ such that 
%$\Upsilon_\downarrow(\sigma)\setminus C\neq\emptyset$. 
%For every $v\in\Upsilon_\downarrow(\sigma)\setminus C$, 
%there exists $X\in\mathbb{P}(\sigma)\setminus\mathbb{L}(\sigma)$ such that 
%\begin{equation*}
%\begin{cases}
%v\in X\\
%\text{and}\\
%X\cap C=\emptyset.
%\end{cases}
%\end{equation*}
%\end{lem}

%\begin{proof}
%We can assume that $C\neq\emptyset$. 
%Let $c\in C$. 
%By the first assertion of Lemma~\ref{L1hat} applied to $V(\sigma)\supsetneq\{v\}=\widehat{\{v\}}$ and $c\in V(\sigma)\setminus\{v\}$, there is 
%$X\in\mathbb{P}(\sigma)\setminus\mathbb{L}(\sigma)$ such that $v\in X$ and $c\not\in X$. 
%As $X\in\mathbb{P}(\sigma)$ and as $v\in X\setminus C$ and $c\in C\setminus X$, 
%we obtain $X\cap C=\emptyset$. 
%\end{proof}

\section{Traverses of a 2-structure}\label{sectiontrav}

Given a 2-structure $\sigma$, a linear order $L$ defined on $V(\sigma)$ is a {\em traverse} of 
$\sigma$ if the following two assertions hold. 
\begin{enumerate}
\item[(A1)] For each $X\in\mathbb{P}(\sigma)$, $X$ is an interval of $L$. 
\item[(A2)] Let $X\in\mathbb{P}(\sigma)\setminus\mathbb{L}(\sigma)$ such that 
$\lambda_\sigma(X)\in E_a(\sigma)$. 
By Assertion A1, $\mathbb{G}(\sigma[X])$ is a factorization of $L[X]$ and we require that 
\begin{equation*}
L[X]/\mathbb{G}(\sigma[X])=O_X\quad\text{(see \eqref{E1linear})}.
\end{equation*}
\end{enumerate}

\begin{prop}[Axiom of Choice]\label{ptraverse}
Any 2-structure admits a traverse.
\end{prop}

\begin{proof}
Consider a 2-structure $\sigma$. 
Using the Axiom of Choice, we associate with each 
$X\in\mathbb{P}(\sigma)\setminus\mathbb{L}(\sigma)$ a linear order $L_X$ defined on 
$\mathbb{G}(\sigma[X])$. 
Moreover, we require that 
$L_X=O_X$ whenever $\lambda_\sigma(X)\in E_a(\sigma)$. 

Now consider the digraph $\Delta$ defined on $V(\sigma)$ as follows. 
Let $v\neq w\in V(\sigma)$. 
By Lemma~\ref{L1tilde}, $\widetilde{\{v,w\}}\in\mathbb{P}(\sigma)\setminus\mathbb{L}(\sigma)$ and there are 
$X_v\neq X_w\in\mathbb{G}(\sigma[\widetilde{\{v,w\}}])$ such that $v\in X_v$ and $w\in X_w$. 
Set 
\begin{equation*}
(v,w)\in A(\Delta)\quad\text{if}\quad X_v<X_w \mod L_{\widetilde{\{v,w\}}}.
\end{equation*}
It is not difficult to verify that $\Delta$ is a linear order which satisfies Assertions~A1 and A2. 
\end{proof}

Using Assertions~A1 and A2, we obtain the following. 

\begin{lem}\label{L1trav}
Given a 2-structure $\sigma$, consider a traverse $T$ of $\sigma$. 
For each $C\in\mathbb{C}(\sigma)$, 
if $\lambda_\sigma(\widetilde{C})\in E_a(\sigma)$, then $C\in\mathbb{C}(T)$. 
\end{lem}

The next result follows from Lemma~\ref{L1trav}. 

\begin{cor}\label{C1trav}
Given a 2-structure $\sigma$, consider a traverse $T$ of $\sigma$. 
For each $C\in\mathscr{L}(\sigma)$, 
$\mathbb{C}(\sigma[C])\subseteq\mathbb{C}(T)$. 
\end{cor}

\begin{proof}
By Proposition~\ref{P1clanmax}, 
$\widetilde{C}\in\mathbb{P}(\sigma)\setminus\mathbb{L}(\sigma)$, 
$\lambda_\sigma(\widetilde{C})\in E_a(\sigma)$ and 
$\{\{c\}:c\in C\}\subseteq\mathbb{G}_1(\sigma[\widetilde{C}])$. 
Let $D\in\mathbb{C}_{\geq 2}(\sigma[C])$. 
We get 
$\{\{d\}:d\in D\}\subseteq\mathbb{G}_1(\sigma[\widetilde{C}])$ and hence $\widetilde{D}=\widetilde{C}$. 
If $D\in\mathbb{P}(\sigma)$, that is, $D=\widetilde{D}$, then $D\in\mathbb{C}(T)$ by Assertion~A1.  
Assume that $D\not\in\mathbb{P}(\sigma)$ so that 
$D\in\mathbb{C}(\sigma)\setminus\mathbb{P}(\sigma)$. 
By Lemma~\ref{L1trav}, $D\in\mathbb{C}(T)$. 
\end{proof}

The following is a simple consequence of Assertion~A1 and Lemma~\ref{L1trav}. 
It generalizes \cite[Theorem 2]{EFHM72}. 

\begin{cor}\label{C2trav}
Let $\sigma$ be an asymmetric 2-structure. 
For a traverse $T$ of $\sigma$, we have $\mathbb{C}(\sigma)\subseteq\mathbb{C}(T)$. 
\end{cor}

The following notion of density is fundamental. 
Let $L$ be a linear order. 
Recall that a subset $W$ of $V(L)$ is {\em dense} if $W$ intersects each interval $I$ of $L$ such that 
$|I|\geq 2$. 
A bicoloration $\beta:V(L)\longrightarrow\{0,1\}$ is {\em dense} if $\beta^{-1}(\{0\})$ and 
$\beta^{-1}(\{1\})$ are dense subsets. 
Beside, given a set $S$, consider $\mathcal{F}\subseteq 2^S$. 
The family $\mathcal{F}$ satisfies the {\em Bernstein property} \cite{EFHM72} if there is $B\subseteq S$ such that for every $X\in\mathcal{F}$, $X$ intersects $B$ and $S\setminus B$. 
Clearly a linear order $L$ admits a dense bicoloration if and only if the family of the intervals of $L$ has the Bernstein property.
We use the following (see the proof of \cite[Theorem~3]{EFHM72}) applied to a traverse. 

\begin{prop}[Axiom of Choice]\label{dense}
Every linear order admits a dense bicoloration. 
\end{prop}

In the next proposition, it is easy to verify that the second assertion implies the first. 
For the converse, see \cite[Lemma~9]{EHM72}. 

\begin{prop}\label{2dense}
The following two assertions are equivalent:
\begin{enumerate}
\item every linear order admits a dense bicoloration;
\item every infinite linear order admits a primitive and faithful 1-extension, that is, a primitive 1-extension which is a tournament. 
\end{enumerate}
\end{prop}

We complete the section with the following three results on primitive bounds. 
The first two follow from Propositions~\ref{dense} and \ref{2dense}. 

\begin{cor}[Erd\H{o}s et al. \cite{EHM72}]\label{C3trav}
For a linear order $L$, we have $p(L)=1$ or $2$. 
Moreover, 
$p(L)=2$ if and only if $L$ is a finite linear order such that $|V(L)|$ is odd. 
\end{cor}

\begin{cor}\label{P1trav}
Consider a reversible 2-structure $\sigma$ such that $\varepsilon(\sigma)\geq 3$. 
Let $a\not\in V(\sigma)$. 
For each $L\in\mathscr{L}(\sigma)$, 
there is faithful extension $\tau$ of $\sigma$ defined on $V(\sigma)\cup\{a\}$ such that $\tau[L\cup\{a\}]$ is primitive. 
\end{cor}

\begin{thm}\label{T2trav}
Let $\sigma$ be an asymmetric 2-structure. 
If $V(\sigma)\not\in\mathbb{L}(\sigma)$, $\lambda_\sigma(V(\sigma))\in E_a(\sigma)$ and 
$\mathbb{G}_{\geq 2}(\sigma)\neq\emptyset$, then $p(\sigma)=1$. 
\end{thm}

\begin{proof}
Denote by $\mathscr{H}$ the smallest clan of $\sigma/\mathbb{G}(\sigma)$ containing 
$\mathbb{G}_{\geq 2}(\sigma)$. 
Consider 
\begin{equation*}
\begin{cases}
\mathscr{H}^-=
\{X\in\mathbb{G}(\sigma):X<Y\!\mod O_{V(\sigma)}\ \text{for every}\ Y\in\mathscr{H}\}\\
\text{and}\\
\mathscr{H}^+=
\{X\in\mathbb{G}(\sigma):X>Y\!\mod O_{V(\sigma)}\ \text{for every}\ Y\in\mathscr{H}\}.
\end{cases}
\end{equation*}

By Proposition~\ref{ptraverse}, $\sigma$ admits a traverse $T$. 
Furthermore $T$ admits a dense bicoloration $\beta$ by Proposition~\ref{dense}. 
We define a bicoloration $\beta^-$ of $T[\bigcup\mathscr{H}^-]$ as follows: $\beta^-=\beta_{\restriction(\bigcup\mathscr{H}^-})$ or 
$(1-\beta)_{\restriction(\bigcup\mathscr{H}^-})$ and, when $O_{V(\sigma)}$ admits a smallest element 
$\{h^-\}$ such that 
$h^-\in\bigcup\mathscr{H}^-$, we require that 
\begin{equation}\label{E0T2trav}
\beta^-(h^-)=1. 
\end{equation}
Similarly, 
$\beta^+$ is a bicoloration of $T[\bigcup\mathscr{H}^+]$ satisfying: $\beta^+=\beta_{\restriction(\bigcup\mathscr{H}^+})$ or 
$(1-\beta)_{\restriction(\bigcup\mathscr{H}^+})$ and, when $O_{V(\sigma)}$ admits a largest element 
$\{h^+\}$ such that 
$h^+\in\bigcup\mathscr{H}^+$, we require that $\beta^+(h^+)=0$. 
Now consider the bicoloration $\beta'$ of $T$ defined as follows. 
For every $v\in V(\sigma)$, 
\begin{equation*}
\beta'(v)=\ 
\begin{cases}
\beta^-(v)\ \text{if}\ v\in\bigcup\mathscr{H}^-\\
\beta(v)\ \text{if}\ v\in\bigcup\mathscr{H}\\
\beta^+(v)\ \text{if}\ v\in\bigcup\mathscr{H}^+.
\end{cases}
\end{equation*}
In general, $\beta'$ may not be a dense bicoloration of $T$. 
But, 
for $H=\bigcup\mathscr{H}^-,\bigcup\mathscr{H}$ or $\bigcup\mathscr{H}^+$, 
\begin{equation}\label{E1aT2trav}
\text{$(\beta')_{\restriction H}$ is a dense bicoloration of $T[H]$.} 
\end{equation}
We also have 
\begin{equation}\label{E2aT2trav}
\text{for each $C\in\mathbb{C}_{\geq 2}(\sigma)$, there are $c,d\in C$ such that $\beta'(c)\neq\beta'(d)$.} 
\end{equation}
Indeed consider $C\in\mathbb{C}_{\geq 2}(\sigma)$. 
For a contradiction, suppose that 
$$|C\cap(\bigcup\mathscr{H}^-)|\leq 1,\quad |C\cap(\bigcup\mathscr{H})|\leq 1\quad\text{and}\quad
|C\cap(\bigcup\mathscr{H}^+)|\leq 1.$$
By Corollary~\ref{C2trav}, $C$ is an interval of $T$. 
Thus, if $C\cap(\bigcup\mathscr{H}^-)\neq\emptyset$ and 
$C\cap(\bigcup\mathscr{H}^+)\neq\emptyset$, then we would have 
$\bigcup\mathscr{H}\subseteq C$ and hence 
$\bigcup\mathbb{G}_{\geq 2}(\sigma)\subseteq C\cap(\bigcup\mathscr{H})$. 
For instance, 
we can assume that $C\cap(\bigcup\mathscr{H}^+)=\emptyset$. 
We get $|C\cap(\bigcup\mathscr{H}^-)|=|C\cap(\bigcup\mathscr{H})|=1$. 
Set $\mathcal{F}_C=\{X\in\mathbb{G}(\sigma):X\cap C\neq\emptyset\}$. 
We obtain $|\mathcal{F}_C\cap\mathscr{H}^-|=|\mathcal{F}_C\cap\mathscr{H}|=1$ and 
$\mathcal{F}_C\cap\mathscr{H}\subseteq\mathbb{G}_1(\sigma)$. 
As $\mathcal{F}_C\in\mathbb{C}(\sigma/\mathbb{G}(\sigma))$ and as 
$\mathcal{F}_C\setminus\mathscr{H}\neq\emptyset$, 
$\mathscr{H}\setminus\mathcal{F}_C\in\mathbb{C}(\sigma/\mathbb{G}(\sigma))$. 
Since $\mathcal{F}_C\cap\mathscr{H}\subseteq\mathbb{G}_1(\sigma)$, 
$\mathbb{G}_{\geq 2}(\sigma)\subseteq\mathscr{H}\setminus\mathcal{F}_C$ which contradicts the minimality of $\mathscr{H}$. 
Consequently there is 
$H=\bigcup\mathscr{H}^-,\bigcup\mathscr{H}$ or $\bigcup\mathscr{H}^+$ such that $|C\cap H|\geq 2$. 
It follows from \eqref{E1aT2trav} that \eqref{E2aT2trav} holds. 

Set $e_0=\lambda_\sigma(V(\sigma))$ and $e_1=\lambda_\sigma(V(\sigma))^\star$. 
Let $a\not\in V(\sigma)$. 
We consider the faithful extension $\tau$ of $\sigma$ defined on 
$V(\sigma)\cup\{a\}$ satisfying 
$$(v,a)_{\tau}=e_{\beta'(v)}\ \text{for every $v\in V(\sigma)$.}$$
We prove that $\tau$ is primitive. 
Let $D\in\mathbb{C}_{\geq 2}(\tau)$. 
It follows from \eqref{E2aT2trav} that 
$a\in D$ and 
\begin{equation}\label{E3aT2trav}
\text{for each $C\in\mathbb{C}_{\geq 2}(\sigma)$, $C\cap(D\setminus\{a\})\neq\emptyset$.} 
\end{equation}
Set $\mathcal{D}=\{X\in\mathbb{G}(\sigma):X\cap(D\setminus\{a\})\neq\emptyset\}$. 
We have $\mathcal{D}\in\mathbb{C}(\sigma/\mathbb{G}(\sigma))$ and 
$\mathbb{G}_{\geq 2}(\sigma)\subseteq\mathcal{D}$ by \eqref{E3aT2trav}. 
By minimality of $\mathscr{H}$, $\mathscr{H}\subseteq\mathcal{D}$. 
Thus 
$\mathscr{H}^-\setminus\mathcal{D},\mathscr{H}^+\setminus\mathcal{D}\in
\mathbb{C}(\sigma/\mathbb{G}(\sigma))$ and hence 
$\bigcup(\mathscr{H}^-\setminus\mathcal{D}),\bigcup(\mathscr{H}^+\setminus\mathcal{D})\in
\mathbb{C}(\sigma)$. 
It follows from \eqref{E3aT2trav} that 
$|\bigcup(\mathscr{H}^-\setminus\mathcal{D})|\leq 1$ and 
$|\bigcup(\mathscr{H}^+\setminus\mathcal{D})|\leq 1$. 
For a contradiction, suppose that 
$|\bigcup(\mathscr{H}^-\setminus\mathcal{D})|=1$ and denote by $h^-$ the unique element of 
$\bigcup(\mathscr{H}^-\setminus\mathcal{D})$. 
Since $\mathscr{H}^-\setminus\mathcal{D}$ is an initial interval of $O_{V(\sigma)}$, 
$\{h^-\}$ is the smallest element of $O_{V(\sigma)}$. 
Given $w\in D\setminus\{a\}$, we have $(h^-,w)_\sigma=\lambda_\sigma(V(\sigma))=e_0$. 
Therefore $(h^-,a)_\tau=e_0$ and hence $\beta'(h^-)=0$. 
We would get $\beta^-(h^-)=0$ contradicting \eqref{E0T2trav}. 
It follows that $\bigcup(\mathscr{H}^-\setminus\mathcal{D})=\emptyset$, that is, 
$\mathscr{H}^-\subseteq\mathcal{D}$. 
Similarly $\mathscr{H}^+\subseteq\mathcal{D}$. 
Consequently $\mathcal{D}=\mathbb{G}(\sigma)$ so that 
$D\setminus\{a\}=V(\sigma)$ and $D=V(\tau)$. 
\end{proof}

\section{Inclusive clans of a 2-structure}\label{sectioninclusive}

Let $\sigma$ be a 2-structure. 
For convenience, set $C(\sigma)=\bigcup\mathscr{C}(\sigma)$, $L(\sigma)=\bigcup\mathscr{L}(\sigma)$ and 
$P(\sigma)=\bigcup\mathscr{P}(\sigma)$. 

Given a 2-structure $\sigma$, a clan $C$ of $\sigma$ is {\em inclusive} if 
$C(\sigma)\cup L(\sigma)\cup P(\sigma)\subseteq C$ and 
$C\cap X\neq\emptyset$ for each $X\in\mathbb{P}_{\geq 2}(\sigma)$. 
The family of the inclusive clans of $\sigma$ is denoted by $\mathbb{I}(\sigma)$. 
Set $\mathbb{I}_{\geq 2}(\sigma)=\{J\in\mathbb{I}(\sigma):|J|\geq 2\}$. 
Of course, $V(\sigma)\in\mathbb{I}(\sigma)$ and $J\neq\emptyset$ for every $J\in\mathbb{I}(\sigma)$. 
Furthermore, for $J\in\mathbb{I}(\sigma)$ and $C\in\mathbb{C}(\sigma)$,
\begin{equation}\label{E1inc}
C\supseteq J\ \Longrightarrow\ C\in\mathbb{I}(\sigma). 
\end{equation}

It is easy to show the next two results. 

\begin{lem}\label{L1inc}
Let $\sigma$ be a 2-structure. 
For each $J\in\mathbb{I}(\sigma)$, $\Upsilon_\downarrow(\sigma)\subseteq J$. 
\end{lem}

\begin{lem}\label{L2inc}
Let $\sigma$ be a 2-structure. 
For each $C\in\mathbb{C}_{\geq 2}(\sigma)$, if $C(\sigma)\cup L(\sigma)\subseteq C$ and if 
$C\cap X\neq\emptyset$ for each $X\in\mathbb{P}_{\geq 2}(\sigma)$, then 
$C\in\mathbb{I}(\sigma)$. 
\end{lem}

\begin{lem}\label{L3inc}
Let $\sigma$ be a 2-structure. 
If there is $v\in V(\sigma)$ such that $\{v\}\in\mathbb{I}(\sigma)$, then 
$\mathfrak{C}(\sigma)=\mathfrak{C}_1(\sigma)$, 
$\Upsilon_\downarrow(\sigma)=\{v\}$ and 
$\{v\}$ is the smallest inclusive clan of $\sigma$ under inclusion. 
\end{lem}

\begin{proof}
By Theorem~\ref{T1equivalence}, 
$\mathfrak{C}_{\geq 2}(\sigma)=\mathscr{C}(\sigma)\cup\mathscr{L}(\sigma)\cup\mathscr{P}(\sigma)$. 
As $C(\sigma)\cup L(\sigma)\cup P(\sigma)\subseteq \{v\}$, we obtain $\mathfrak{C}(\sigma)=\mathfrak{C}_1(\sigma)$. 
Now, we prove that $v\in\Upsilon_\downarrow(\sigma)$. 
Suppose for a contradiction that $|\widehat{\{v\}}|\geq 2$. 
By Lemma~\ref{L1hat}, 
$\{v\}\in\mathbb{G}(\sigma[\widehat{\{v\}}])$. 
Moreover, as $\{v\}\cap X\neq\emptyset$ for each $X\in\mathbb{G}_{\geq 2}(\sigma[\widehat{\{v\}}])$, we obtain 
$\mathbb{G}_{\geq 2}(\sigma[\widehat{\{v\}}])=\emptyset$ that is 
$\mathbb{G}(\sigma[\widehat{\{v\}}])=\mathbb{G}_{1}(\sigma[\widehat{\{v\}}])$. 
It follows from Proposition~\ref{P1clanmax} that 
$\widehat{\{v\}}\in\mathscr{C}(\sigma)\cup\mathscr{L}(\sigma)\cup\mathscr{P}(\sigma)$. 
By Theorem~\ref{T1equivalence}, $\widehat{\{v\}}\in\mathfrak{C}_{\geq 2}(\sigma)$ contradicting $\mathfrak{C}(\sigma)=\mathfrak{C}_1(\sigma)$. 
Consequently $\widehat{\{v\}}=\{v\}$ that is $v\in\Upsilon_\downarrow(\sigma)$. 
It follows from Lemma~\ref{L1inc} that $\Upsilon_\downarrow(\sigma)=\{v\}$. 
Moreover, 
$\Upsilon_\downarrow(\sigma)\subseteq J$ for each $J\in\mathbb{I}(\sigma)$. 
Therefore $\{v\}$ is the smallest inclusive clan of $\sigma$ under inclusion. 
\end{proof}

\begin{lem}\label{L4inc}
Let $\sigma$ be a 2-structure. 
Given $J\in\mathbb{I}(\sigma)$, we have $J\cap C\neq\emptyset$ for every 
$C\in\mathbb{C}_{\geq 2}(\sigma)$. 
\end{lem}

\begin{proof}
By definition of an inclusive clan, we consider $C\in\mathbb{C}_{\geq 2}(\sigma)\setminus\mathbb{P}(\sigma)$. 
By Lemma~\ref{L1nonprime}, 
there is a non trivial clan $F$ of $\sigma[\widetilde{C}]/\mathbb{G}(\sigma[\widetilde{C}])$ such that $C=\bigcup F$. 
If there exists $X\in F\cap\mathbb{G}_{\geq 2}(\sigma[\widetilde{C}])$, then $J\cap X\neq\emptyset$ and hence 
$J\cap C\neq\emptyset$. 
Assume that $F\subseteq \mathbb{G}_{1}(\sigma[\widetilde{C}])$. 
We obtain that $v\simeq_\sigma w$ for any $v,w\in C$. 
Thus there is $D\in\mathfrak{C}_{\geq 2}(\sigma)$ such that $D\supseteq \bigcup F=C$. 
By Theorem~\ref{T1equivalence}, 
$D\in\mathscr{C}(\sigma)\cup\mathscr{L}(\sigma)\cup\mathscr{P}(\sigma)$ so that $C\subseteq J$. 
\end{proof}

\begin{prop}\label{P1inc}
Let $\sigma$ be a 2-structure. 
For any $J,K\in\mathbb{I}(\sigma)$, 
$J\cap K\in\mathbb{I}(\sigma)$ and $J\cup K\in\mathbb{I}(\sigma)$. 
\end{prop}

\begin{proof}
To begin, we show that $J\cap K\neq\emptyset$. 
First, assume that $|J|=|K|=1$. 
It follows from Lemma~\ref{L3inc} that $J=K=\Upsilon_\downarrow(\sigma)$. 
Second, assume for example that $|J|\geq 2$. 
By Lemma~\ref{L4inc}, $J\cap K\neq\emptyset$. 
It follows that $J\cup K\in\mathbb{C}(\sigma)$ and hence 
$J\cup K\in\mathbb{I}(\sigma)$ by \eqref{E1inc}. 
Lastly, we prove that $J\cap K\in\mathbb{I}(\sigma)$. 
As $C(\sigma)\cup L(\sigma)\cup P(\sigma)\subseteq J$ and $C(\sigma)\cup L(\sigma)\cup P(\sigma)\subseteq K$, 
we have $C(\sigma)\cup L(\sigma)\cup P(\sigma)\subseteq J\cap K$. 
Now consider $X\in\mathbb{P}_{\geq 2}$. 
We have $J\cap X\neq\emptyset$ and $K\cap X\neq\emptyset$. 
Thus $J,X$ and $K,X$ are comparable under inclusion. 
We obtain $J\cap K\subseteq X$ or $X\subseteq J\cap K$. 
\end{proof}

The next result provides a structural analysis of non prime and inclusive clans. 

\begin{prop}\label{P2inc}
Let $\sigma$ be a 2-structure. 
Given $J\in\mathbb{I}(\sigma)$, 
if $J\not\in\mathbb{P}(\sigma)$, then one of the following two assertions holds.
\begin{enumerate}
\item $\lambda_\sigma(\widetilde{J})\in E_s(\sigma)$ and we have: 
$|\mathbb{G}_{1}(\sigma[\widetilde{J}])|=1$, 
$|\mathbb{G}_{\geq 2}(\sigma[\widetilde{J}])|\geq 2$ and 
$J=\bigcup\mathbb{G}_{\geq 2}(\sigma[\widetilde{J}])$. 
\item $\lambda_\sigma(\widetilde{J})\in E_a(\sigma)$, 
$\mathbb{G}_{\geq 2}(\sigma[\widetilde{J}])\neq\emptyset$ and, by denoting by 
$\mathcal{I}$ the smallest interval of $O_{\widetilde{J}}$ containing 
$\mathbb{G}_{\geq 2}(\sigma[\widetilde{J}])$ and by considering 
\begin{equation*}
\begin{cases}
\mathcal{I}^-=\{\{j\}\in\mathbb{G}(\sigma[\widetilde{J}])\setminus\mathcal{I}:\{j\}<X\hspace{-2mm}\mod O_{\widetilde{J}}\ \ \text{for every}\ X\in\mathcal{J}\}\\
\text{and}\\
\mathcal{I}^+=\{\{j\}\in\mathbb{G}(\sigma[\widetilde{J}])\setminus\mathcal{I}:\{j\}>X\hspace{-2mm}\mod O_{\widetilde{J}}\ \ \text{for every}\ X\in\mathcal{J}\},
\end{cases}
\end{equation*}
we have: 
\begin{itemize}
\item $\bigcup\mathcal{I}\subseteq J$; 
\item if $|\mathcal{I}^-|\geq 2$, then $|\mathcal{I}^+|=1$ and 
$\bigcup\mathcal{I}^+=\widetilde{J}\setminus J=\max O_{\widetilde{J}}$;
\item if $|\mathcal{I}^+|\geq 2$, then $|\mathcal{I}^-|=1$ and 
$\bigcup\mathcal{I}^-=\widetilde{J}\setminus J=\min O_{\widetilde{J}}$;
\item for each $j\in\widetilde{J}\setminus J$, 
$\{j\}=\min O_{\widetilde{J}}$ or $\max O_{\widetilde{J}}$;
\item $|\widetilde{J}\setminus J|=1$ or $2$; 
\item if $|\mathbb{G}_{\geq 2}(\sigma[\widetilde{J}])|=1$, then 
$|\widetilde{J}\setminus J|=1$ and $|\mathbb{G}_{1}(\sigma[\widetilde{J}])|\geq 2$. 
\end{itemize}
\end{enumerate}
\end{prop}

\begin{proof}
Let $J\in\mathbb{I}(\sigma)\setminus\mathbb{P}(\sigma)$. 
By Lemma~\ref{L1nonprime}, $\widetilde{J}\in\mathbb{P}(\sigma)\setminus\mathbb{L}(\sigma)$ and there is a nontrivial clan $\mathcal{F}$ of $\sigma[\widetilde{J}]/\mathbb{G}(\sigma[\widetilde{J}])$ such that 
$J=\bigcup\mathcal{F}$. 
Let $X\in\mathbb{G}_{\geq 2}(\sigma[\widetilde{J}])$. 
As $X\in\mathbb{P}_{\geq 2}(\sigma)$, 
$J\cap X\neq\emptyset$ and hence $X\in\mathcal{F}$. 
Therefore $\mathbb{G}_{\geq 2}(\sigma[\widetilde{J}])\subseteq\mathcal{F}$ and 
$\bigcup\mathbb{G}_{\geq 2}(\sigma[\widetilde{J}])\subseteq J$. 
Since $\mathcal{F}$ is a nontrivial clan of 
$\sigma[\widetilde{C}]/\mathbb{G}(\sigma[\widetilde{C}])$, 
$\lambda_\sigma(\widetilde{J})\in E_s(\sigma)\cup E_a(\sigma)$ and we distinguish the following two cases. 

First, assume that $\lambda_\sigma(\widetilde{J})\in E_s(\sigma)$. 
Since $J\neq\widetilde{J}$ and $\bigcup\mathbb{G}_{\geq 2}(\sigma[\widetilde{J}])\subseteq J$, 
we have $\mathbb{G}_{1}(\sigma[\widetilde{J}])\neq\emptyset$. 
Suppose for a contradiction that $|\mathbb{G}_{1}(\sigma[\widetilde{J}])|\geq 2$. 
By Proposition~\ref{P1clanmax}, 
$\bigcup\mathbb{G}_{1}(\sigma[\widetilde{J}])\in\mathscr{C}(\sigma)$. 
Therefore $\bigcup\mathbb{G}_{1}(\sigma[\widetilde{J}])\subseteq J$ and we would have $J=\widetilde{J}$ because $\bigcup\mathbb{G}_{\geq 2}(\sigma[\widetilde{J}])\subseteq J$. 
It follows that $|\mathbb{G}_{1}(\sigma[\widetilde{J}])|=1$. 
Consequently 
$\mathbb{G}_{\geq 2}(\sigma[\widetilde{J}])=\mathcal{F}$. 

Second, assume that $\lambda_\sigma(\widetilde{J})\in E_a(\sigma)$. 
Suppose for a contradiction that $\mathbb{G}_{\geq 2}(\sigma[\widetilde{J}])=\emptyset$. 
It follows from Proposition~\ref{P1clanmax} that $\widetilde{J}\in\mathscr{L}(\sigma)$ which would imply that $J=\widetilde{J}$. 
Thus that $\mathbb{G}_{\geq 2}(\sigma[\widetilde{J}])\neq\emptyset$. 
As $\mathbb{G}_{\geq 2}(\sigma[\widetilde{J}])\subseteq\mathcal{F}$ and as $\mathcal{F}$ is an interval of $O_{\widetilde{J}}$, we obtain $\mathcal{I}\subseteq\mathcal{F}$ and hence 
$\bigcup\mathcal{I}\subseteq J$. 

Assume that $|\mathcal{I}^-|\geq 2$. 
Since $\mathcal{I}$ is the smallest interval of $O_{\widetilde{J}}$ containing 
$\mathbb{G}_{\geq 2}(\sigma[\widetilde{J}])$, $\mathcal{I}^-$ is a maximal interval of singletons of 
$O_{\widetilde{J}}$. 
By Proposition~\ref{P1clanmax}, $\bigcup\mathcal{I}^-\in\mathscr{L}(\sigma)$ and hence 
$\bigcup\mathcal{I}^-\subseteq J$. 
Thus 
\begin{equation}\label{E2P2inc}
\begin{cases}
|\mathcal{I}^-|\geq 2\ \Longrightarrow\ \bigcup\mathcal{I}^-\subseteq J\\
\text{and similarly}\\
|\mathcal{I}^+|\geq 2\ \Longrightarrow\ \bigcup\mathcal{I}^+\subseteq J.
\end{cases}
\end{equation}
Assume again that $|\mathcal{I}^-|\geq 2$. 
By \eqref{E2P2inc}, $\bigcup\mathcal{I}^-\subseteq J$. 
As $J\subsetneq\widetilde{J}$ and $\bigcup\mathcal{I}\subseteq J$, we get 
$\bigcup\mathcal{I}^+\not\subseteq J$. 
It follows from \eqref{E2P2inc} that $|\mathcal{I}^+|=1$. 
Clearly $\bigcup\mathcal{I}^+=\max O_{\widetilde{J}}$. 
Since $(\bigcup\mathcal{I}^-)\cup(\bigcup\mathcal{I})\subseteq J$ and $J\subsetneq\widetilde{J}$, 
we have $\bigcup\mathcal{I}^+=\widetilde{J}\setminus J$. 
Therefore 
\begin{equation}\label{E3P2inc}
\begin{cases}
|\mathcal{I}^-|\geq 2\ \Longrightarrow\ |\mathcal{I}^+|=1\ \text{and}\ 
\bigcup\mathcal{I}^+=\widetilde{J}\setminus J=\max O_{\widetilde{J}}\\
\text{and similarly}\\
|\mathcal{I}^+|\geq 2\ \Longrightarrow\ |\mathcal{I}^-|=1\ \text{and}\ 
\bigcup\mathcal{I}^-=\widetilde{J}\setminus J=\min O_{\widetilde{J}}. 
\end{cases}
\end{equation} 

Now consider $j\in\widetilde{J}\setminus J$. 
As $\bigcup\mathcal{I}\subseteq J$, $j\in(\bigcup\mathcal{I}^-)\cup(\bigcup\mathcal{I}^+)$. 
For instance, assume that $j\in\bigcup\mathcal{I}^-$. 
It follows from \eqref{E2P2inc} that $|\mathcal{I}^-|\leq 1$ so that 
$\mathcal{I}^-=\{\{j\}\}$. 
Clearly $\{j\}=\min O_{\widetilde{J}}$. 
In the same manner, if $j\in\bigcup\mathcal{I}^+$, then $\{j\}=\max O_{\widetilde{J}}$. 

In particular, it follows that 
\begin{equation}\label{E4P2inc}
\text{$|\widetilde{J}\setminus J|=1$ or $2$.}
\end{equation}
Assume that $|\widetilde{J}\setminus J|=2$. 
By what precedes, there are $j^-,j^+\in\widetilde{J}\setminus J$ such that 
$\{j^-\}=\min O_{\widetilde{J}}$, $\{j^+\}=\max O_{\widetilde{J}}$ and 
$\widetilde{J}\setminus J=\{j^-,j^+\}$. 
We get $(\bigcup\mathcal{I}^-)\cup(\bigcup\mathcal{I}^+)=\widetilde{J}\setminus J$ so that 
$\bigcup\mathcal{I}=J$. 
Since $J\not\in\mathbb{P}(\sigma)$, $|\mathcal{I}|\geq 2$ and hence 
$|\mathbb{G}_{\geq 2}(\sigma[\widetilde{J}])|\geq 2$. 
Thus 
\begin{equation}\label{E5P2inc}
|\widetilde{J}\setminus J|=2\ \Longrightarrow\ |\mathbb{G}_{\geq 2}(\sigma[\widetilde{J}])|\geq 2.
\end{equation}

Lastly, assume that $|\mathbb{G}_{\geq 2}(\sigma[\widetilde{J}])|=1$. 
It follows from \eqref{E4P2inc} and \eqref{E5P2inc} that $|\widetilde{J}\setminus J|=1$. 
Furthermore we clearly have $\mathbb{G}_{\geq 2}(\sigma[\widetilde{J}])=\mathcal{I}$. 
As $J\not\in\mathbb{P}(\sigma)$, we get 
$\bigcup\mathcal{I}\subsetneq J\subsetneq\widetilde{J}$. 
Therefore $|\mathbb{G}_1(\sigma[\widetilde{J}])|\geq 2$. 
\end{proof}

\begin{thm}\label{T1inc}
Let $\sigma$ be a 2-structure. 
Given $J,K\in\mathbb{I}(\sigma)$, if $J\setminus K\neq\emptyset$ and $K\setminus J\neq\emptyset$, then 
the following holds
\begin{itemize}
\item $\widetilde{J}=\widetilde{K}$, $\widetilde{J}\in\mathbb{P}(\sigma)\setminus\mathbb{L}(\sigma)$ and $\lambda_\sigma(\widetilde{J})\in E_a(\sigma)$; 
\item the linear order $O_{\widetilde{J}}$ admits a smallest element $\min O_{\widetilde{J}}$ and a largest element 
$\max O_{\widetilde{J}}$ which belong to $\mathbb{G}_{1}(\sigma[\widetilde{J}])$;
\item $\mathbb{G}_{\geq 2}(\sigma[\widetilde{J}])\neq\emptyset$, $\mathbb{G}(\sigma[\widetilde{J}])\setminus\{\min O_{\widetilde{J}},\max O_{\widetilde{J}}\}$ is the smallest interval of 
$O_{\widetilde{J}}$ containing $\mathbb{G}_{\geq 2}(\sigma[\widetilde{J}])$ and 
$\{J,K\}=\{\widetilde{J}\setminus\min O_{\widetilde{J}},\widetilde{J}\setminus\max O_{\widetilde{J}}\}$;
\item 
for each $H\in\mathbb{I}(\sigma)\setminus\{J,K\}$, either 
$H\subseteq J\cap K$ or $J\cup K\subseteq H$. 
\end{itemize}
\end{thm}

\begin{proof}
It follows from Proposition~\ref{P1inc} that $J\cap K\neq\emptyset$. 
Therefore $J,K\not\in\mathbb{P}(\sigma)$ and 
we can apply Proposition~\ref{P2inc} to $J$ and $K$ as follows. 

As $J\cap K\neq\emptyset$, we have $\widetilde{J}\cap\widetilde{K}\neq\emptyset$ and hence 
$\widetilde{J}\subseteq\widetilde{K}$ or $\widetilde{K}\subseteq\widetilde{J}$. 
Assume that $\widetilde{J}\subseteq\widetilde{K}$. 
For a contradiction, suppose that $\widetilde{J}\subsetneq\widetilde{K}$. 
Since $K\in\mathbb{C}(\sigma)\setminus\mathbb{P}(\sigma)$, 
$\widetilde{K}\in\mathbb{P}(\sigma)\setminus\mathbb{L}(\sigma)$ by Lemma~\ref{L1nonprime}. 
By maximality of the elements of $\mathbb{G}(\sigma[\widetilde{K}])$, there exists $X\in\mathbb{G}(\sigma[\widetilde{K}])$ such that $X\supseteq\widetilde{J}$. 
Since $J\not\in\mathbb{P}(\sigma)$, 
$|J|\geq 2$ and hence 
$X\in\mathbb{G}_{\geq 2}(\sigma[\widetilde{K}])$. 
It follows from Proposition~\ref{P2inc} that $X\subseteq K$ so that we would obtain $J\subseteq\widetilde{J}\subseteq X\subseteq K$ contradicting $J\setminus K\neq\emptyset$. 
Therefore $\widetilde{J}=\widetilde{K}$. 

If $\lambda_\sigma(\widetilde{J})\in E_s(\sigma)$, then we would obtain by Proposition~\ref{P2inc} that 
$J=K=\bigcup\mathbb{G}_{\geq 2}(\sigma[\widetilde{J}])$. 
Thus $\lambda_\sigma(\widetilde{J})\in E_a(\sigma)$. 
It follows from Proposition~\ref{P2inc} that 
for each $j\in\widetilde{J}\setminus (J\cap K)$, 
$\{j\}=\min O_{\widetilde{J}}$ or $\max O_{\widetilde{J}}$. 
Consequently 
$O_{\widetilde{J}}$ admits a smallest element $\min O_{\widetilde{J}}$ and a largest element 
$\max O_{\widetilde{J}}$ which belong to $\mathbb{G}_{1}(\sigma[\widetilde{J}])$ 
and, by interchanging $J$ and $K$ if necessary, we have 
$J=\widetilde{J}\setminus\min O_{\widetilde{J}}$ and $K=\widetilde{J}\setminus\max O_{\widetilde{J}}$. 

By Proposition~\ref{P2inc}, $\mathbb{G}_{\geq 2}(\sigma[\widetilde{J}])\neq\emptyset$. 
As in Proposition~\ref{P2inc}, denote by 
$\mathcal{I}$ the smallest interval of $O_{\widetilde{J}}$ containing 
$\mathbb{G}_{\geq 2}(\sigma[\widetilde{J}])$ and consider 
\begin{equation*}
\begin{cases}
\mathcal{I}^-=\{\{j\}\in\mathbb{G}(\sigma[\widetilde{J}])\setminus\mathcal{I}:\{j\}<X\hspace{-2mm}\mod O_{\widetilde{J}}\ \ \text{for every}\ X\in\mathcal{J}\}\\
\text{and}\\
\mathcal{I}^+=\{\{j\}\in\mathbb{G}(\sigma[\widetilde{J}])\setminus\mathcal{I}:\{j\}>X\hspace{-2mm}\mod O_{\widetilde{J}}\ \ \text{for every}\ X\in\mathcal{J}\}.
\end{cases}
\end{equation*}
As $J=\widetilde{J}\setminus\min O_{\widetilde{J}}$, it follows from 
Proposition~\ref{P2inc} that $|\mathcal{I}^-|\leq 1$. 
Thus $\mathcal{I}^-=\{\min O_{\widetilde{J}}\}$. 
Similarly we get $\mathcal{I}^+=\{\max O_{\widetilde{J}}\}$. 
Therefore 
$\mathcal{I}=\mathbb{G}(\sigma[\widetilde{J}])\setminus\{\min O_{\widetilde{J}},$ $\max O_{\widetilde{J}}\}$. 

Lastly, consider $H\in\mathbb{I}(\sigma)\setminus\{J,K\}$. 
By Lemma~\ref{L4inc}, $H\cap J\neq\emptyset$. 
Thus $H\cap\widetilde{J}\neq\emptyset$ and hence either $H\subsetneq\widetilde{J}$ or $\widetilde{J}\subseteq H$. 
In the second instance, we have $J\cup K\subseteq H$. 
In the first, 
by considering $\mathcal{H}=\{X\in\mathbb{G}(\sigma[\widetilde{J}]):H\cap X\neq\emptyset\}$, we have 
$\mathcal{H}\in\mathcal{C}(\sigma[\widetilde{J}]/\mathbb{G}(\sigma[\widetilde{J}]))$. 
As $\mathcal{H}\neq\mathbb{G}(\sigma[\widetilde{J}])$, 
$\min O_{\widetilde{J}}\not\in\mathcal{H}$ or 
$\max O_{\widetilde{J}}\not\in\mathcal{H}$. 
For instance, assume that $\min O_{\widetilde{J}}\not\in\mathcal{H}$ so that $H\subsetneq J$. 
Since $H\in\mathbb{I}(\sigma)$, we have $\mathbb{G}_{\geq 2}(\sigma[\widetilde{J}])\subseteq\mathcal{H}$. 
As $\mathcal{H}\in\mathcal{C}(\sigma[\widetilde{J}]/\mathbb{G}(\sigma[\widetilde{J}]))$, that is, $\mathcal{H}$ is an interval of 
$O_{\widetilde{J}}$, we obtain $\mathcal{I}\subseteq\mathcal{H}$ and hence 
$$\mathbb{G}(\sigma[\widetilde{J}])\setminus\{\min O_{\widetilde{J}},\max O_{\widetilde{J}}\}
\subseteq\mathcal{H}\subseteq\mathbb{G}(\sigma[\widetilde{J}])\setminus\{\min O_{\widetilde{J}}\}.$$ 
Suppose for a contradiction that $\max O_{\widetilde{J}}\in\mathcal{H}$. 
We get $\mathcal{H}=\mathbb{G}(\sigma[\widetilde{J}])\setminus\{\min O_{\widetilde{J}}\}$. 
Thus $|\mathcal{H}|\geq 2$ and 
we would have $H=\bigcup\mathcal{H}=J$. 
Consequently $\max O_{\widetilde{J}}\not\in\mathcal{H}$ and 
$$\mathcal{H}=
\mathbb{G}(\sigma[\widetilde{J}])\setminus\{\min O_{\widetilde{J}},\max O_{\widetilde{J}}\}.$$
Therefore $H\subseteq\bigcup\mathcal{H}=\widetilde{J}\setminus(\min O_{\widetilde{J}}\cup\max O_{\widetilde{J}})=J\cap K$. 
\end{proof}

The next is an immediate consequence of Theorem~\ref{T1inc}. 

\begin{cor}\label{C1inc}
If $\sigma$ is a symmetric 2-structure, then $(\mathbb{I}(\sigma),\subseteq)$ is a linear order.
\end{cor}

We complete the section with a result on primitive bounds of asymmetric 2-structures.

\begin{thm}\label{T2inc}
Given an asymmetric 2-structure $\sigma$, if 
$\mathbb{I}(\sigma)\setminus\{V(\sigma)\}\neq\emptyset$, then $p(\sigma)=1$. 
\end{thm}

\begin{proof}
As $\mathbb{I}(\sigma)\setminus\{V(\sigma)\}\neq\emptyset$, we have $\nu(\sigma)\geq 3$. 
It follows from Lemma~\ref{L3inc} and \eqref{E1inc} that 
there is $J\in\mathbb{I}(\sigma)\setminus\{V(\sigma)\}$ such that $|J|\geq 2$. 
In particular, $\sigma$ is imprimitive and hence 
\begin{equation}\label{E0T2inc}
p(\sigma)\geq 1.
\end{equation}
By Proposition~\ref{ptraverse}, $\sigma$ admits a traverse $T$. 
Furthermore $T$ admits a dense bicoloration $\beta$ by Proposition~\ref{dense}. 
Given $e\in E(\sigma)$, set $e_0=e$ and $e_1=e^\star$. 
Let $a\not\in V(\sigma)$. 
We associate with $\beta$ the faithful extension $\tau_{\beta}$ of $\sigma$ defined on 
$V(\sigma)\cup\{a\}$ satisfying 
\begin{subequations}\label{E1T2inc}
\begin{gather}
(v,a)_{\tau_{\beta}}=e_{\beta(v)}\ \text{for every $v\in J$}\label{E1T2inc1};\\
(v,a)_{\tau_{\beta}}=((v,J)_\sigma)^\star\ \text{for every $v\in V(\sigma)\setminus J$.}\label{E1T2inc2}
\end{gather}
\end{subequations}
We establish the following. 
If $\tau_{\beta}$ is imprimitive, then 
\begin{itemize}
\item $V(\sigma)\in\mathbb{P}(\sigma)\setminus\mathbb{L}(\sigma)$ and 
$\lambda_\sigma(V(\sigma))\in E_a(\sigma)$;
\item $O_{V(\sigma)}$ has a smallest element $\min O_{V(\sigma)}$ and a largest 
$\max O_{V(\sigma)}$;
\item There are $v_{{\rm min}},v_{{\rm max}}\in V(\sigma)$ such that 
$\min O_{V(\sigma)}=\{v_{{\rm min}}\}$ and 
$\max O_{V(\sigma)}=\{v_{{\rm max}}\}$;
\item $J=V(\sigma)\setminus\{v\}$ where $v\in\{v_{{\rm min}},v_{{\rm max}}\}$;
\item $V(\tau)\setminus\{u\}$ is the unique nontrivial clan of $\tau_\beta$ where 
$u\in\{v_{{\rm min}},v_{{\rm max}}\}\setminus\{v\}$ and 
\begin{equation}\label{E2T2inc}
e_{\beta(u)}=(u,v)_{\sigma}. 
\end{equation}
\end{itemize}

Assume that $\tau_{\beta}$ is imprimitive and consider a nontrivial clan $D_{\beta}$ of $\tau_{\beta}$. 
Suppose for a contradiction that $a\not\in D_{\beta}$. 
We get $D_{\beta}\in\mathbb{C}_{\geq 2}(\sigma)$ and hence 
$J\cap D_{\beta}\in\mathbb{C}(\sigma)$. 
By Corollary~\ref{C2trav}, $J\cap D_{\beta}$ is an interval of $T$. 
For $u,v\in J\cap D_{\beta}$, we have $(u,a)_{\tau_{\beta}}=(v,a)_{\tau_{\beta}}$. 
It follows from \eqref{E1T2inc1} that $e_{\beta(u)}=e_{\beta(v)}$ and hence $\beta(u)=\beta(v)$. 
Therefore $\beta_{\restriction J\cap D_{\beta}}$ is constant. 
Since $D_{\beta}\in\mathbb{C}_{\geq 2}(\sigma)$, it follows from 
Lemma~\ref{L4inc} that $J\cap D_{\beta}\neq\emptyset$. 
By density of $\beta$, 
$|J\cap D_{\beta}|=1$. 
Denote by $x$ the unique element of $J\cap D_{\beta}$. 
We show that $D_{\beta}\setminus\{x\}\in\mathbb{C}(\sigma)$. 
It suffices to verify that 
$D_{\beta}\setminus\{x\}\in\mathbb{C}(\sigma[D_{\beta}])$, that is, $x\longleftrightarrow_\sigma D_{\beta}\setminus\{x\}$. 
Let $u,v\in D_{\beta}\setminus\{x\}$. 
Since $D_{\beta}\in\mathbb{C}(\tau_{\beta})$, we have $(u,a)_{\tau_{\beta}}=(v,a)_{\tau_{\beta}}$. 
It follows from \eqref{E1T2inc2} that 
$(u,x)_\sigma=(v,x)_\sigma$. 
Consequently $D_{\beta}\setminus\{x\}\in\mathbb{C}(\sigma)$. 
It follows from Lemma~\ref{L4inc} that $|D_{\beta}\setminus\{x\}|=1$. 
Hence $|D_{\beta}|=2$ and $D_{\beta}$ is linear. 
By Lemma~\ref{L1clanmax}, we obtain that $D_{\beta}\subseteq L(\sigma)$ and $J$ would not be an inclusive clan of $\sigma$. 
It follows that 
\begin{equation}\label{E3T2inc}
a\in D_{\beta}.
\end{equation}
We get 
$D_{\beta}\setminus\{a\}\in\mathbb{C}(\sigma)$. 
Suppose for a contradiction that $J\cap(D_{\beta}\setminus\{a\})=\emptyset$. 
We have 
$(D_{\beta}\setminus\{a\})\longleftrightarrow_\sigma J$. 
It would follow that $a\longleftrightarrow_{\tau_{\beta}} J$ which contradicts the density of $\beta$. 
Thus $J\cap(D_{\beta}\setminus\{a\})\neq\emptyset$ and it follows from \eqref{E1T2inc2} that 
\begin{equation}\label{E4T2inc}
J\cup(D_{\beta}\setminus\{a\})=V(\sigma).
\end{equation}
Since $J\neq V(\sigma)$, 
it follows from \eqref{E4T2inc} that 
\begin{equation}\label{E5T2inc}
(D_{\beta}\setminus\{a\})\setminus J\neq\emptyset. 
\end{equation}
Furthermore, as $D_{\beta}$ is a nontrivial clan of $\tau_{\beta}$, it follows from 
\eqref{E3T2inc} and \eqref{E4T2inc} that 
\begin{equation}\label{E6T2inc}
J\setminus(D_{\beta}\setminus\{a\})\neq\emptyset. 
\end{equation}
Thus 
$(D_{\beta}\setminus\{a\})\setminus J\in\mathbb{C}(\sigma)$. 
It follows from Lemma~\ref{L4inc} that 
$|(D_{\beta}\setminus\{a\})\setminus J|\leq 1$. 
Hence $|(D_{\beta}\setminus\{a\})\setminus J|=1$ by \eqref{E5T2inc}. 
Denote by $v$ the unique element of $(D_{\beta}\setminus\{a\})\setminus J$. 
It follows from \eqref{E4T2inc} that $$J=V(\sigma)\setminus\{v\}.$$
Let $C\in\mathbb{C}_{\geq 2}(\sigma)\setminus\{V(\sigma)\}$ such that $C\ni v$. 
Clearly $J\setminus C\neq\emptyset$ and 
$C\setminus J\neq\emptyset$. 
Furthermore $J\cap C\neq\emptyset$ by Lemma~\ref{L4inc}. 
Therefore $C\not\in\mathbb{P}(\sigma)$. 
It follows that $\{v\}\in\mathbb{G}(\sigma)$ so that 
$V(\sigma)\in\mathbb{P}(\sigma)\setminus\mathbb{L}(\sigma)$. 
We get 
$\mathbb{G}(\sigma)\setminus\{\{v\}\}=\{X\in\mathbb{G}(\sigma):J\cap X\neq\emptyset\}\in\mathbb{C}(\sigma/\mathbb{G}(\sigma))$. 
Therefore $\lambda_\sigma(V(\sigma))\neq\varpi$. 
As $\sigma$ is asymmetric, $\lambda_\sigma(V(\sigma))\in E_a(\sigma)$ and 
$\{v\}=\min O_{V(\sigma)}$ or $\max O_{V(\sigma)}$. 
For convenience, assume that 
$\{v\}=\min O_{V(\sigma)}$. 
By \eqref{E5T2inc}, 
$J\setminus(D_{\beta}\setminus\{a\})\in\mathbb{C}(\sigma)$ so that 
$(D_{\beta}\setminus\{a\})\longleftrightarrow_\sigma J\setminus(D_{\beta}\setminus\{a\})$. 
As $D_\beta\in\mathbb{C}(\tau_\beta)$, we get 
$a\longleftrightarrow_\sigma J\setminus(D_{\beta}\setminus\{a\})$. 
By density of $\beta$, $|J\setminus(D_{\beta}\setminus\{a\})|\leq 1$. 
Thus $|J\setminus(D_{\beta}\setminus\{a\})|=1$ by \eqref{E6T2inc}. 
Denote by $u$ the unique element of $J\setminus(D_{\beta}\setminus\{a\})$. 
It follows from \eqref{E3T2inc} and \eqref{E4T2inc} that $$D_\beta=V(\tau_\beta)\setminus\{u\}.$$ 
Hence $D_\beta\setminus\{a\}\in\mathbb{C}(\sigma)$ and 
$V(\sigma)\setminus(D_\beta\setminus\{a\})=\{u\}$. 
As shown for $v$, we obtain that 
$\{u\}\in\mathbb{G}(\sigma)$ and 
$\{u\}=\min O_{V(\sigma)}$ or $\max O_{V(\sigma)}$. 
Since $u\in J$, $u\neq v$ and hence 
$\{u\}=\max O_{V(\sigma)}$. 
As $D_\beta\in\mathbb{C}(\tau_\beta)$, $(u,a)_{\tau_\beta}=(u,v)_\sigma$. 
Thus \eqref{E2T2inc} follows from \eqref{E1T2inc1}. 

We conclude as follows. 
Clearly  \eqref{E2T2inc} cannot hold for both dense bicolorations $\beta$ and $1-\beta$ of $T$. 
Consequently $\tau_\beta$ or $\tau_{1-\beta}$ is primitive. 
Thus $p(\sigma)\leq 1$. 
By \eqref{E0T2inc}, $p(\sigma)=1$. 
\end{proof}

\section{Primitive bounds of non reversible 2-structures}\label{per2s}

Given $n>0$, $L_n$ denotes the usual linear order on $\{0,\ldots,n\}$. 
When $n$ is even, there does not exist a primitive tournament which is a 1-extension of $L_n$. 
On the other hand, there is a primitive tournament which is a 2-extension of $L_n$ (see \cite{M72}). 

The 2-structure $\sigma(L_n)$ is defined on $\{0,\ldots,n\}$ by $E(\sigma(L_n))=\{e,e^\star\}$ where 
$e=\{(p,q):0\leq p<q\leq n\}$. 
Let $\tau$ be an extension of $\sigma(L_n)$ to $\{0,\ldots,n+1\}$ satisfying
\begin{itemize}
\item for $0\leq i\leq n-1$ such that $i$ is even, $(n+1,i)\equiv_\tau (0,1)$ and $(i,n+1)\equiv_\tau (1,0)$;
\item for $0\leq i\leq n-1$ such that $i$ is odd, $(i,n+1)\equiv_\tau (0,1)$ and $(n+1,i)\equiv_\tau (1,0)$;
\item $(n,n+1)\equiv_\tau (n+1,n)$. 
\end{itemize}
It is simple to verify that $\tau$ is primitive. 
Moreover, $\tau$ is not identifiable with a tournament because $(n,n+1)\equiv_\tau (n+1,n)$. 
If $(n,n+1)\not\equiv_\tau (0,1)$ and $(n,n+1)\not\equiv_\tau (1,0)$, then $\varepsilon(\tau)=3$. 
Since $\varepsilon(\sigma(L_n))=2$, $\sigma(L_n)\hookrightarrow\tau$ is not bijective, that is, \eqref{E1faith} does not hold. 
Now assume for instance that $(n,n+1)\equiv_\tau (0,1)$. 
We obtain $(n,n+1)\in(\sigma(L_n)\hookrightarrow\tau)(e)$. 
We have also $(n+1,n)\equiv_\tau (0,1)$ and hence $(n+1,n)\in(\sigma(L_n)\hookrightarrow\tau)(e)$. 
Thus $$(n,n+1)\in(\sigma(L_n)\hookrightarrow\tau)(e)\cap ((\sigma(L_n)\hookrightarrow\tau)(e))^\star.$$
But $e\cap e^\star=\emptyset$. 
Therefore \eqref{E2faith} does not hold. 

Conditions \eqref{E1faith} and \eqref{E2faith} ensure that the faithful extensions of a 2-structure $\sigma$ are 2-structures of the 
``same type'' as $\sigma$. 
Also, they ensure that the faithful extensions of a reversible 2-structure are reversible as well. 

\begin{lem}\label{L1faith}
Let $\sigma$ be a 2-structure with $\varepsilon(\sigma)\geq 2$. Given a faithful extension $\tau$ of $\sigma$, $\sigma$ is reversible if and only if $\tau$ is. 
\end{lem}

\begin{proof}
To begin, assume that $\sigma$ is reversible. 
Consider $e\in E(\sigma)$. 
Let $(u,v)\in((\sigma\hookrightarrow\tau)(e))^\star$. 
By \eqref{E1faith}, there is $f_{(u,v)}\in E(\sigma)$ such that $(u,v)_\tau=(\sigma\hookrightarrow\tau)(f_{(u,v)})$. 
We obtain $(u,v)\in(\sigma\hookrightarrow\tau)(f_{(u,v)})\cap ((\sigma\hookrightarrow\tau)(e))^\star$. 
By \eqref{E2faith}, $f_{(u,v)}\cap e^\star\neq\emptyset$. 
As $\sigma$ is reversible, $e^\star\in E(\sigma)$ and hence $f_{(u,v)}=e^\star$. 
It follows that $(u,v)_\tau=(\sigma\hookrightarrow\tau)(e^\star)$ for every $(u,v)\in((\sigma\hookrightarrow\tau)(e))^\star$. 
Thus 
\begin{equation}\label{E1L1faith}
((\sigma\hookrightarrow\tau)(e))^\star\subseteq(\sigma\hookrightarrow\tau)(e^\star).
\end{equation}
By applying \eqref{E1L1faith} to $e^\star\in E(\sigma)$, we obtain 
$((\sigma\hookrightarrow\tau)(e^\star))^\star\subseteq(\sigma\hookrightarrow\tau)(e)$ and hence 
$(\sigma\hookrightarrow\tau)(e^\star)\subseteq((\sigma\hookrightarrow\tau)(e))^\star$. 
Therefore 
\begin{equation}\label{E2L1faith}
(\sigma\hookrightarrow\tau)(e^\star)=((\sigma\hookrightarrow\tau)(e))^\star
\end{equation} 
for every $e\in E(\sigma)$. 
Now consider $e_\tau\in E(\tau)$. 
By \eqref{E1faith}, there is $e\in E(\sigma)$ such that $(\sigma\hookrightarrow\tau)(e)=e_\tau$. 
By \eqref{E2L1faith}, $(e_\tau)^\star=(\sigma\hookrightarrow\tau)(e^\star)$. 
Consequently, 
$(e_\tau)^\star\in E(\tau)$ for each $e_\tau\in E(\tau)$. 
It follows that $\tau$ is reversible. 

Conversely, assume that $\tau$ is reversible. 
We observe the following. 
It follows from \eqref{E1faith} that  
\begin{equation}\label{E3L1faith}
(\sigma\hookrightarrow\tau)(e)\cap((V(\sigma)\times V(\sigma))\setminus\{(v,v):v\in V(\sigma)\})=e
\end{equation} for each $e\in E(\sigma)$. 
Consider $e\in E(\sigma)$. It suffices to show that $e^\star\in E(\sigma)$. 
Since $\tau$ is reversible and since $(\sigma\hookrightarrow\tau)(e)\in E(\tau)$, we have 
$((\sigma\hookrightarrow\tau)(e))^\star\in E(\tau)$. 
By \eqref{E1faith}, there is $f\in E(\sigma)$ such that $((\sigma\hookrightarrow\tau)(e))^\star=(\sigma\hookrightarrow\tau)(f)$. 
Hence 
\begin{equation*}
\begin{array}{rl}
((\sigma\hookrightarrow\tau)(e))^\star\cap&\hspace{-3mm}((V(\sigma)\times V(\sigma))\setminus\{(v,v):v\in V(\sigma)\})\\
=&\hspace{-3mm}(\sigma\hookrightarrow\tau)(f)\cap((V(\sigma)\times V(\sigma))\setminus\{(v,v):v\in V(\sigma)\}). 
\end{array}
\end{equation*}
By \eqref{E3L1faith}, $(\sigma\hookrightarrow\tau)(f)\cap((V(\sigma)\times V(\sigma))\setminus\{(v,v):v\in V(\sigma)\})=f$. 
Furthermore 
\begin{equation*}
\begin{array}{rl}
&\hspace{-3mm}((\sigma\hookrightarrow\tau)(e))^\star\cap((V(\sigma)\times V(\sigma))\setminus\{(v,v):v\in V(\sigma)\})\\
=&\hspace{-3mm}((\sigma\hookrightarrow\tau)(e)\cap((V(\sigma)\times V(\sigma))\setminus\{(v,v):v\in V(\sigma)\}))^\star\\ 
=&\hspace{-3mm}e^\star\hspace{5mm}\text{by \eqref{E3L1faith}}.
\end{array}
\end{equation*}
Therefore $e^\star=f$ and hence $e^\star\in E(\sigma)$. 
\end{proof}

\begin{lem}\label{L2faith}
Consider 2-structures $\sigma$ and $\tau$ such that $\tau$ is an extension of $\sigma$. 
If $\sigma$ and $\tau$ are reversible, then 
\eqref{E2faith} holds. 
\end{lem}

\begin{proof}
Consider $e,f\in E(\sigma)$ such that $(\sigma\hookrightarrow\tau)(e)\cap((\sigma\hookrightarrow\tau)(f))^\star\neq\emptyset$. 
As $(\sigma\hookrightarrow\tau)(f)\supseteq f$, $((\sigma\hookrightarrow\tau)(f))^\star\supseteq f^\star$. 
Since $\sigma$ is reversible, $f^\star\in E(\sigma)$ and we have $(\sigma\hookrightarrow\tau)(f^\star)\supseteq f^\star$. 
Since $\tau$ is reversible, $((\sigma\hookrightarrow\tau)(f))^\star\in E(\tau)$. 
It follows that $((\sigma\hookrightarrow\tau)(f))^\star=(\sigma\hookrightarrow\tau)(f^\star)$. 
As $\sigma\hookrightarrow\tau$ is injective, $(\sigma\hookrightarrow\tau)(e)\cap((\sigma\hookrightarrow\tau)(f))^\star\neq\emptyset$ implies that $e=f^\star$. 
\end{proof}

Given a 2-structure $\sigma$, recall that $\sigma\land\sigma^\star$ is reversible, 
$E(\sigma\land\sigma^\star)=\{e\cap f^\star:e,f\in E(\sigma),e\cap f^\star\neq\emptyset\}$  and 
$\mathbb{C}(\sigma\land\sigma^\star)=\mathbb{C}(\sigma)$. 

\begin{prop}\label{P1nonrev}
Let $\sigma$ be a 2-structure.
\begin{enumerate}
\item If $\tau$ is a faithful extension of $\sigma$, then $\tau\land\tau^\star$ is a faithful extension of $\sigma\land\sigma^\star$.
\item If $\rho$ is a faithful extension of $\sigma\land\sigma^\star$, then there is a faithful extension $\tau$ of $\sigma$ such that 
$\rho=\tau\land\tau^\star$. 
\end{enumerate}
\end{prop}

\begin{proof}
First, assume that $\tau$ is a faithful extension of $\sigma$. 
Clearly $\tau\land\tau^\star$ is an extension of $(\sigma\land\sigma^\star$. 
Consider an element $e\cap f^\star$ of $E(\sigma\land\sigma^\star)$ where $e,f\in E(\sigma)$ such that 
$e\cap f^\star\neq\emptyset$. 
As $(\sigma\hookrightarrow\tau)(e)\supseteq e$ and $((\sigma\hookrightarrow\tau)(f))^\star\supseteq f^\star$, we have 
$(\sigma\hookrightarrow\tau)(e)\cap ((\sigma\hookrightarrow\tau)(f))^\star\supseteq e\cap f^\star$. 
Furthermore $(\sigma\hookrightarrow\tau)(e)\cap ((\sigma\hookrightarrow\tau)(f))^\star\in E(\tau\land\tau^\star)$. 
Thus 
\begin{equation}\label{E1P1nonrev}
((\sigma\land\sigma^\star)\hookrightarrow(\tau\land\tau^\star))(e\cap f^\star)=
(\sigma\hookrightarrow\tau)(e)\cap ((\sigma\hookrightarrow\tau)(f))^\star
\end{equation} 
for any $e,f\in E(\sigma)$ such that $e\cap f^\star\neq\emptyset$. 

Now we prove that \eqref{E1faith} holds for the extension $\tau\land\tau^\star$ of 
$\sigma\land\sigma^\star$. 
It suffices to show that $(\sigma\land\sigma^\star)\hookrightarrow(\tau\land\tau^\star)$ is surjective. 
Consider $e_{\tau\land\tau^{\star}}\in E(\tau\land\tau^{\star})$. 
There are $e_\tau,f_\tau\in E(\tau)$ such that $e_{\tau\land\tau^{\star}}=e_\tau\cap(f_\tau)^{\star}$. 
As $\sigma\hookrightarrow\tau$ is bijective, there exist 
$e_\sigma,f_\sigma\in E(\sigma)$ such that $e_\tau=(\sigma\hookrightarrow\tau)(e_\sigma)$ and 
$f_\tau=(\sigma\hookrightarrow\tau)(f_\sigma)$. 
Thus $e_{\tau\land\tau^{\star}}=
(\sigma\hookrightarrow\tau)(e_\sigma)\cap((\sigma\hookrightarrow\tau)(f_\sigma))^{\star}$. 
Since (1.2) is satisfied by the extension $\tau$ of $\sigma$, we obtain 
$e_\sigma\cap (f_\sigma)^{\star}\neq\emptyset$. 
It follows from \eqref{E1P1nonrev} that 
$e_{\tau\land\tau^{\star}}=
((\sigma\land\sigma^\star)\hookrightarrow(\tau\land\tau^\star))(e_\sigma\cap (f_\sigma)^\star)$. 

Lastly, 
since $\sigma\land\sigma^\star$ and $\tau\land\tau^\star$ are reversible, \eqref{E2faith} holds by Lemma~\ref{L2faith}. 

Second, assume that $\rho$ is a faithful extension of $\sigma\land\sigma^\star$. 
For each $e\in E(\sigma)$, set 
\begin{equation}\label{E2P1nonrev}
\underline{e}=\bigcup_{\{f\in E(\sigma):e\cap f^\star\neq\emptyset\}}((\sigma\land\sigma^\star)\hookrightarrow\rho)(e\cap f^\star).
\end{equation}
Since $E(\rho)=\{((\sigma\land\sigma^\star)\hookrightarrow\rho)(e\cap f^\star):e,f\in E(\sigma),e\cap f^\star\neq\emptyset\}$ is a partition of $(V(\rho)\times V(\rho))\setminus\{(v,v):v\in V(\rho)\}$, 
$\{\underline{e}:e\in E(\sigma)\}$ is also. 
Denote by $\tau$ the unique 2-structure defined on $V(\tau)=V(\rho)$ by 
$E(\tau)=\{\underline{e}:e\in E(\sigma)\}$. 

Let $e\in E(\sigma)$. 
By \eqref{E2P1nonrev}, 
$$\underline{e}\cap(V(\sigma)\times V(\sigma))=\bigcup_{\{f\in E(\sigma):e\cap f^\star\neq\emptyset\}}((\sigma\land\sigma^\star)\hookrightarrow\rho)(e\cap f^\star)\cap(V(\sigma)\times V(\sigma)).$$
As $\rho$ is an extension of $\sigma\land\sigma^\star$, we get 
$((\sigma\land\sigma^\star)\hookrightarrow\rho)(e\cap f^\star)\cap(V(\sigma)\times V(\sigma))=
e\cap f^\star$ for each $f\in E(\sigma)$ such that $e\cap f^\star\neq\emptyset$. 
Thus 
$$\underline{e}\cap(V(\sigma)\times V(\sigma))=\bigcup_{\{f\in E(\sigma):e\cap f^\star\neq\emptyset\}}
e\cap f^\star=e.$$
Therefore $\tau$ is an extension of $\sigma$ such that 
$(\sigma\hookrightarrow\tau)(e)=\underline{e}$. 
It follows that \eqref{E1faith} holds for the extension $\tau$ of $\sigma$. 
For \eqref{E2faith}, consider $e,g\in E(\sigma)$ such that 
$(\sigma\hookrightarrow\tau)(e)\cap ((\sigma\hookrightarrow\tau)(g))^\star\neq\emptyset$, that is, 
$\underline{e}\cap(\underline{g})^\star\neq\emptyset$. 
By \eqref{E2P1nonrev}, there exist $f\in E(\sigma)$ such that $e\cap f^\star\neq\emptyset$ and 
$h\in E(\sigma)$ such that $g\cap h^\star\neq\emptyset$ satisfying 
$$((\sigma\land\sigma^\star)\hookrightarrow\rho)(e\cap f^\star)\cap(((\sigma\land\sigma^\star)\hookrightarrow\rho)(g\cap h^\star))^\star\neq\emptyset.$$
As \eqref{E2faith} holds for the extension $\rho$ of $\sigma\land\sigma^\star$, we obtain 
$(e\cap f^\star)\cap(g\cap h^\star)^\star\neq\emptyset$ and hence $e\cap g^\star\neq\emptyset$. 
\end{proof}

The next is an immediate consequence of Proposition~\ref{P1nonrev} and of the fact that 
$\mathbb{C}(\sigma\land\sigma^\star)=\mathbb{C}(\sigma)$ for a 2-structure $\sigma$. 

\begin{cor}\label{C1nonrev}
Let $\sigma$ be a 2-structure.
\begin{enumerate}
\item If $\tau$ is a primitive faithful extension of $\sigma$, then $\tau\land\tau^\star$ is a primitive faithful extension of $\sigma\land\sigma^\star$.
\item If $\rho$ is a primitive faithful extension of $\sigma\land\sigma^\star$, then there is a primitive faithful extension $\tau$ of $\sigma$ such that $\rho=\tau\land\tau^\star$. 
\end{enumerate}
\end{cor}

We conclude with the following. 

\begin{thm}\label{T1nonrev}
For every 2-structure $\sigma$, $p(\sigma)=p(\sigma\land\sigma^\star)$. 
\end{thm}

\section{Primitive bounds of reversible 2-structures}

We begin with a remark on the construction of faithful extensions of reversible 2-structures. 
Consider reversible 2-structures $\sigma$ and $\sigma'$ such that $V(\sigma)\cap V(\sigma')=\emptyset$ with 
$\varepsilon_a(\sigma')\leq\varepsilon_a(\sigma)$ and 
$\varepsilon_s(\sigma')\leq\varepsilon_s(\sigma)$. 
There exists an injection $\iota:E(\sigma')\longrightarrow E(\sigma)$ such that 
$\iota(E_a(\sigma'))\subseteq E_a(\sigma)$ and $\iota(E_s(\sigma'))\subseteq E_s(\sigma)$. 
Consider also a function $\varphi:V(\sigma)\times V(\sigma')\longrightarrow E(\sigma)$. 
It is easy to verify that there is a unique extension $\tau$ of $\sigma$ and $\sigma'$ defined on 
$V(\sigma)\cup V(\sigma')$ which is a faithful extension of $\sigma$ such that 
$\sigma'\hookrightarrow\tau=(\sigma\hookrightarrow\tau)\circ\iota$ and 
$(v,v')_\tau=((\sigma\hookrightarrow\tau)\circ\varphi)(v,v')$ for every $(v,v')\in V(\sigma)\times V(\sigma')$. 

\subsection{Preliminary results}

\begin{lem}\label{L2ext}
Consider a primitive and reversible 2-structure $\sigma$. 
Let $a\not\in V(\sigma)$. 
If $V(\sigma)$ is finite, then 
$\sigma$ admits 
$\varepsilon(\sigma)^{\nu(\sigma)}-\varepsilon(\sigma)\nu(\sigma)-\varepsilon(\sigma)$ 
primitive and faithful extensions of $\sigma$ to $V(\sigma)\cup\{a\}$. 
If $V(\sigma)$ is infinite, then 
$\sigma$ admits 
$\varepsilon(\sigma)^{\nu(\sigma)}$ such extensions. 
\end{lem}

\begin{proof}
Consider a faithful extension $\tau$ of $\sigma$ to $V(\sigma)\cup\{a\}$. 
Assume that $\tau$ is imprimitive and consider a nontrivial clan $C$ of $\tau$. 
As $\sigma$ is primitive, we have either $|C\cap V(\sigma)|\leq 1$ or 
$C\cap V(\sigma)=V(\sigma)$. 
Since $C$ is a nontrivial clan of $\tau$, $|C|\geq 2$ and hence $|C\cap V(\sigma)|\geq 1$. 
By the same, $C\subsetneq V(\tau)$. 
Therefore we obtain that either there is $v\in V(\sigma)$ such that $C=\{a,v\}$ or 
$C=V(\sigma)$. 
In the second instance, $\overrightarrow{\tau}(a)$ is constant. 
In the first, $\overrightarrow{\tau}(a)_{\restriction V(\sigma)\setminus\{v\}}=
\overrightarrow{\sigma}(v)$. 
Conversely, if $\overrightarrow{\tau}(a)$ is constant, then $V(\sigma)$ is a nontrivial clan of $\tau$. 
Furthermore, if there is $v\in V(\sigma)$ such that 
$\overrightarrow{\tau}(a)_{\restriction V(\sigma)\setminus\{v\}}=
\overrightarrow{\sigma}(v)$, then $\{a,v\}$ is a nontrivial clan of $\tau$. 

Consequently, we have: 
given a faithful extension $\tau$ of $\sigma$ to $V(\sigma)\cup\{a\}$, 
$\tau$ is imprimitive if and only if either $\overrightarrow{\tau}(a)$ is constant or 
there is $v\in V(\sigma)$ such that 
$\overrightarrow{\tau}(a)_{\restriction V(\sigma)\setminus\{v\}}=
\overrightarrow{\sigma}(v)$. 
Therefore $\sigma$ admits $\varepsilon(\sigma)\nu(\sigma)+\varepsilon(\sigma)$ 
imprimitive and faithful extensions to $V(\sigma)\cup\{a\}$. 
If $V(\sigma)$ is finite, then $\varepsilon(\sigma)$ is also and 
$\sigma$ admits $\varepsilon(\sigma)^{\nu(\sigma)}-\varepsilon(\sigma)\nu(\sigma)-\varepsilon(\sigma)$ primitive and faithful extensions of $\sigma$ to $V(\sigma)\cup\{a\}$. 
Assume that $V(\sigma)$ is infinite. 
We get $\varepsilon(\sigma)\leq\nu(\sigma)$ and hence 
$\varepsilon(\sigma)\nu(\sigma)+\varepsilon(\sigma)=\nu(\sigma)<
\varepsilon(\sigma)^{\nu(\sigma)}$. 
Thus 
$\sigma$ admits $\varepsilon(\sigma)^{\nu(\sigma)}$ primitive and faithful extensions of $\sigma$ to $V(\sigma)\cup\{a\}$. 
\end{proof}

\begin{lem}\label{L1ext}
Let $\sigma$ be a reversible 2-structure such that $\varepsilon(\sigma)\geq 2$. 
Consider $S\subseteq V(\sigma)$ such that $2\leq|S|<\aleph_0$ and $S$ is $e$-complete where $e\in E_s(\sigma)$. 
Let $S'$ be a set such that $S'\cap V(\sigma)=\emptyset$ and 
$|S'|=\mathfrak{log}_{\varepsilon(\sigma)}(|S|+1)$. 
There exists a faithful extension $\tau$ of $\sigma$ defined on $V(\sigma)\cup S'$ satisfying
\begin{enumerate}
\item for every $s'\in S'$, there is $s\in S$ such that $(\overrightarrow{\tau}(s))^{-1}(\{e\})\cap S'=\{s'\}$;
\item $\tau[S\cup S']$ is primitive. 
\end{enumerate}
\end{lem}

\begin{proof}
Since $S$ is finite, $S'$ is also and $\varepsilon(\sigma)^{|S'|-1}\leq|S|<\varepsilon(\sigma)^{|S'|}$. 
As $\varepsilon(\sigma)^{|S'|-1}\geq|S'|$, we get $|S'|\leq|S|$. 
Thus there exists an injection $\varphi:S'\longrightarrow S$. 
Since $\varepsilon(\sigma)\geq 2$, there exists $f\in E(\sigma)\setminus\{e\}$. 
Consider $A:\varphi(S')\longrightarrow E(\sigma)^{S'}\setminus\{\overline{e}_{\restriction S'}\}$ defined as follows. 
 For each $s'\in S'$, 
\begin{equation*}
\begin{array}{rccl}
A(\varphi(s')):&S'&\longrightarrow&E(\sigma)\\
&s'&\longmapsto&e\\
&t'\in S'\setminus\{s'\}&\longmapsto&f.
\end{array}
\end{equation*}
The function $A$ is injective. 
Since $|S|+1\leq\varepsilon(\sigma)^{|S'|}$, 
there exists an injection $B:S\longrightarrow E(\sigma)^{S'}\setminus\{\overline{e}_{\restriction S'}\}$ such that $B_{\restriction \varphi(S')}=A$. 
We consider a faithfull extension $\tau$ of $\sigma$ defined on $V(\sigma)\cup S'$ satisfying 
\begin{equation*}
\begin{cases}
\text{for any $s'\neq t'\in S'$, $(s',t')_\tau=f$ or $f^\star$,}\\
\text{and}\\
\text{for each $s\in S$, $\overrightarrow{\tau}(s)_{\restriction S'}=B(s)$.}
\end{cases}
\end{equation*}
For every $s'\in S'$, $(\overrightarrow{\tau}(\varphi(s')))^{-1}(\{e\})\cap S'=\{s'\}$, and 
it is simple to verify that $\tau[S\cup S']$ is primitive. 
\end{proof}

\begin{prop}\label{P1ext}
Let $\sigma$ be a reversible 2-structure such that $\varepsilon(\sigma)\geq 2$. 
Consider a primitive, reversible and infinite 2-structure $\sigma'$ such that 
$V(\sigma)\cap V(\sigma')=\emptyset$, 
$\varepsilon_a(\sigma')\leq\varepsilon_a(\sigma)$, $\varepsilon_s(\sigma')\leq\varepsilon_s(\sigma)$ 
and $\varepsilon(\sigma)<\varepsilon(\sigma)^{\nu(\sigma')}$.

For each $S\subseteq V(\sigma)$ such that 
$|S|\leq\varepsilon(\sigma)^{\nu(\sigma')}$, there exists an extension $\tau$ of $\sigma$ and $\sigma'$ to $V(\sigma)\cup V(\sigma')$ such that $\tau$ is a faithful extension of $\sigma$ and $\tau[S\cup V(\sigma')]$ is primitive. 
\end{prop}

\begin{proof}
Since $\varepsilon_a(\sigma')\leq\varepsilon_a(\sigma)$ and 
$\varepsilon_s(\sigma')\leq\varepsilon_s(\sigma)$, there exists an injection 
$\iota:E(\sigma')\longrightarrow E(\sigma)$ such that 
$\iota(E_a(\sigma'))\subseteq E_a(\sigma)$, 
$\iota(E_s(\sigma'))\subseteq E_s(\sigma)$ and $\iota(e^\star)=\iota(e)^\star$ 
for every $e\in E_a(\sigma')$. 
In what follows, we identify $e'\in E(\sigma')$ with $\iota(e')\in E(\sigma)$. 

Let $v'\in V(\sigma')$. 
We have 
$\overrightarrow{\sigma'}(v'):V(\sigma')\setminus\{v'\}\longrightarrow E(\sigma)$. 
Denote by $\mathcal{F}_{v'}$ the family of the extensions of 
$\overrightarrow{\sigma'}(v')$ to $V(\sigma')$. 
Also set $\overline{e}:V(\sigma')\longrightarrow\{e\}$ for each $e\in E(\sigma)$. 
For each $v'\in V(\sigma')$, we have $|\mathcal{F}_{v'}|=\varepsilon(\sigma)$. 
Thus 
$$|(\bigcup_{v'\in V(\sigma')}\mathcal{F}_{v'})\cup\{\overline{e}:e\in E(\sigma)\}|\leq
\nu(\sigma')\varepsilon(\sigma)+\varepsilon(\sigma)=\max(\varepsilon(\sigma),\nu(\sigma')).$$
Furthermore, as $\varepsilon(\sigma)\geq 2$, $\varepsilon(\sigma)^{\nu(\sigma')}>\nu(\sigma')$. 
Since $\varepsilon(\sigma)^{\nu(\sigma')}>\varepsilon(\sigma)$ by hypothesis, we get 
$\varepsilon(\sigma)^{\nu(\sigma')}>\max(\varepsilon(\sigma),\nu(\sigma')).$ 
Therefore 
$$|E(\sigma)^{V(\sigma')}\setminus
((\bigcup_{v'\in V(\sigma')}\mathcal{F}_{v'})\cup\{\overline{e}:e\in E(\sigma)\})|=\varepsilon(\sigma)^{\nu(\sigma')}.$$
As $|S|\leq\varepsilon(\sigma)^{\nu(\sigma')}$, 
there exists an injection 
$$A:S\longrightarrow E(\sigma)^{V(\sigma')}\setminus 
((\bigcup_{v'\in V(\sigma')}\mathcal{F}_{v'})\cup\{\overline{e}:e\in E(\sigma)\}).$$
Consider the faithful extension $\tau$ of $\sigma$ defined on $V(\sigma)\cup V(\sigma')$ as follows. 
Let $e\in E(\sigma)$. 
Given $u\neq v\in V(\sigma)\cup V(\sigma')$, 
\begin{equation*}
(u,v)_\tau=\ 
\begin{cases}
\text{$(u,v)_\sigma$ if $u,v\in V(\sigma)$,}\\
\text{$(u,v)_{\sigma'}$ if $u,v\in V(\sigma')$,}\\
\text{$A(u)(v)$ if $u\in S$ and $v\in V(\sigma')$,}\\
\text{$e$ if $u\in V(\sigma)\setminus S$ and $v\in V(\sigma')$.}
\end{cases}
\end{equation*}
It is not difficult to verify that $\tau[S\cup V(\sigma')]$ is primitive. 
\end{proof}

The rather technical appearing conditions of the following two results permit their use in different ways throughout the next subsection. 
The usefulness of dense bicolorations of traverses appears in the first one.

\begin{lem}\label{L0bound}
Let $\sigma$ be a reversible 2-structure such that $\varepsilon(\sigma)\geq 2$. 
Consider a traverse $T$ of $\sigma$ and a dense bicoloration $\beta$ of $T$. 
Consider a set $S'$ such that $S'\neq\emptyset$ and $S'\cap V(\sigma)=\emptyset$. 
Let $A_0\neq A_1\in E(\sigma)^{S'}$. 
Also let $\mathscr{F}\subseteq\mathscr{L}(\sigma)\cup\mathscr{P}(\sigma)$. 
Consider a faithful extension $\tau$ of $\sigma$ to $V(\sigma)\cup S'$ satisfying
\begin{subequations}\label{E0L0bound}
\begin{align}
&\text{for each $C\in\mathscr{C}(\sigma)$, 
$\mathbb{C}_{\geq 2}(\sigma[C])\cap\mathbb{C}(\tau)=\emptyset$;}\label{E1L0bound}\\
&\text{for each $X\in\mathscr{F}$, 
$\tau[X\cup S']$ is primitive;}\label{E2L0bound}\\
&\text{and}\nonumber\\
&\text{for each $v\in V(\sigma)\setminus(C(\sigma)\cup (\bigcup\mathscr{F}))$, 
$\overrightarrow{\tau}(v)_{\restriction S'}=A_{\beta(v)}$.}\label{E3L0bound}
\end{align}
\end{subequations}
Then, we have 
\begin{equation*}
\mathbb{C}_{\geq 2}(\sigma)\cap\mathbb{C}(\tau)=\emptyset.
\end{equation*}
\end{lem}

\begin{proof}
To begin, given 
$C\in\mathscr{C}(\sigma)\cup\mathscr{L}(\sigma)\cup\mathscr{P}(\sigma)$, 
we prove that 
\begin{equation}\label{E4L0bound}
\mathbb{C}_{\geq 2}(\sigma[C])\cap\mathbb{C}(\tau)=\emptyset. 
\end{equation}
Clearly \eqref{E4L0bound} follows from \eqref{E1L0bound} when 
$C\in\mathscr{C}(\sigma)$. 
Assume that 
$C\in\mathscr{L}(\sigma)\cup\mathscr{P}(\sigma)$ and consider 
$D\in\mathbb{C}_{\geq 2}(\sigma[C])$. 
We have to prove that $D\not\in\mathbb{C}(\tau)$. 
If $C\in\mathscr{F}$, then 
$\tau[C\cup S']$ is primitive by \eqref{E2L0bound}. 
Thus $D\not\in\mathbb{C}(\tau[C\cup S'])$ and hence $D\not\in\mathbb{C}(\tau)$. 

Assume that $C\in(\mathscr{L}(\sigma)\cup\mathscr{P}(\sigma))\setminus\mathscr{F}$. 
We show that $D$ is an interval of $T$. 
By Corollary~\ref{C1trav}, if $C\in\mathscr{L}(\sigma)$, then $D$ is an interval of 
$T$. 
If $C\in\mathscr{P}(\sigma)$, then $D=C$ and $C\in\mathbb{P}(\sigma)$ because $\sigma[C]$ is primitive. 
By Assertion~A1, $D$ is an interval of $T$. 

By \eqref{E3L0bound}, since $D\subseteq V(\sigma)\setminus(C(\sigma)\cup (\bigcup\mathscr{F}))$, we have 
$\overrightarrow{\tau}(v)_{\restriction S'}=A_{\beta(v)}$ for every $v\in D$. 
As $\beta$ is a dense bicoloraton of $T$, 
there exist $u\neq v\in D$ such that $\beta(u)\neq\beta(v)$. 
Therefore $A_{\beta(u)}\neq A_{\beta(v)}$ and 
$\overrightarrow{\tau}(u)_{\restriction S'}\neq
\overrightarrow{\tau}(v)_{\restriction S'}$ by \eqref{E3L0bound}. 
Consequently there is $s'\in S'$ such that 
$(u,s')_\tau\neq(v,s')_\tau$ so that $D\not\in\mathbb{C}(\tau)$. 
It follows that \eqref{E4L0bound} holds. 

To continue we show that 
\begin{equation}\label{E5L0bound}
\mathbb{P}_{\geq 2}(\sigma)\cap\mathbb{C}(\tau)=\emptyset. 
\end{equation}
Given $X\in\mathbb{P}_{\geq 2}(\sigma)$, we distinguish the following two cases to show that $X\not\in\mathbb{C}(\tau)$. 

First, assume that there exists $C\in\mathscr{C}(\sigma)\cup\mathscr{L}(\sigma)\cup\mathscr{P}(\sigma)$ such that $X\cap C\neq\emptyset$. 
As $X\in\mathbb{P}(\sigma)$, $X\subseteq C$ or $C\subseteq X$. 
Thus $X\cap C=X$ or $C$ and hence $X\cap C\in\mathbb{C}_{\geq 2}(\sigma[C])$. 
It follows from \eqref{E4L0bound} that there is $s'\in S'$ such that $s'\not\longleftrightarrow_\sigma X\cap C$. 
Therefore $s'\not\longleftrightarrow_\sigma X$ and $X\not\in\mathbb{C}(\tau)$. 

Second, assume that $X\cap(C(\sigma)\cup L(\sigma)\cup P(\sigma))=\emptyset$. 
By Assertion~A1, $X$ is an interval of $T$. 
Since $\beta$ is a dense bicoloration of $T$, there are $x\neq y\in X$ such that $\beta(x)\neq\beta(y)$. 
Thus $A_{\beta(x)}\neq A_{\beta(y)}$ and 
$\overrightarrow{\tau}(x)_{\restriction S'}\neq\overrightarrow{\tau}(y)_{\restriction S'}$. 
by \eqref{E3L0bound}. 
Consequently there is $s'\in S'$ such that 
$(x,s')_\tau\neq(y,s')_\tau$ so that 
$X\not\in\mathbb{C}(\tau)$. 
It follows that \eqref{E5L0bound} holds. 

To conclude, consider $C\in\mathbb{C}_{\geq 2}(\sigma)$. 
We must show that $C\not\in\mathbb{C}(\tau)$. 
By \eqref{E5L0bound}, assume that 
$C\not\in\mathbb{P}(\sigma)$. 
By Lemma~\ref{L1nonprime}, $\widetilde{C}\in\mathbb{P}(\sigma)\setminus\mathbb{L}(\sigma)$ and 
$C=\bigcup\mathcal{F}$ where $\mathcal{F}$ is a nontrivial clan of $\sigma[\widetilde{C}]/\mathbb{G}(\sigma[\widetilde{C}])$. 
We distinguish the following two cases. 

First, assume that there exists $X\in\mathcal{F}\cap\mathbb{G}_{\geq 2}(\sigma[\widetilde{C}])$. 
We have $X\in\mathbb{P}_{\geq 2}(\sigma)$ and 
$X\not\in\mathbb{C}(\tau)$ by \eqref{E5L0bound}. 
As $X\subseteq C$, we obtain $C\not\in\mathbb{C}(\tau)$. 

Second, assume that $\mathcal{F}\subseteq\mathbb{G}_{1}(\sigma[\widetilde{C}])$. 
Hence $\{v\}\in\mathbb{G}(\sigma[\widetilde{C}])$ for every $v\in C$. 
By maximality of elements of $\mathbb{G}(\sigma[\widetilde{C}])$, $\widehat{\{v\}}=\widetilde{C}$ for every $v\in C$. 
We obtain that $v\simeq_\sigma w$ for any $v,w\in C$. 
Thus there is $D\in\mathfrak{C}_{\geq 2}(\sigma)$ such that $D\supseteq C$. 
By Theorem~\ref{T1equivalence}, $D\in\mathscr{C}(\sigma)\cup\mathscr{L}(\sigma)\cup\mathscr{P}(\sigma)$. 
By \eqref{E4L0bound}, $C\not\in\mathbb{C}(\tau)$. 
\end{proof}

The following is a simple consequence of Lemma~\ref{L0bound} when only 1-extensions are considered. 
The notion of an inclusive clans follows from it. 

\begin{cor}\label{C0bound}
Let $\sigma$ be a reversible 2-structure such that $\varepsilon(\sigma)\geq 2$. 
Consider a traverse $T$ of $\sigma$ and a dense bicoloration $\beta$ of $T$. 
Let $a\not\in V(\sigma)$. 
Consider $e_0\neq e_1\in E(\sigma)$. 
Also let $\mathscr{F}\subseteq\mathscr{L}(\sigma)\cup\mathscr{P}(\sigma)$. 
Consider a faithful extension $\tau$ of $\sigma$ to $V(\sigma)\cup\{a\}$ satisfying
\begin{itemize}
\item for each $C\in\mathscr{C}(\sigma)$, 
$\mathbb{C}_{\geq 2}(\sigma[C])\cap\mathbb{C}(\tau)=\emptyset$;
\item for each $X\in\mathscr{F}$, 
$\tau[X\cup\{a\}]$ is primitive;
\item for each $v\in V(\sigma)\setminus(C(\sigma)\cup (\bigcup\mathscr{F}))$, 
$(v,a)_{\tau}=e_{\beta(v)}$.
\end{itemize}
Then, the following holds for each $D_\tau\in\mathbb{C}_{\geq 2}(\tau)$
\begin{enumerate}
\item $a\in D_\tau$;
\item for every $C\in\mathbb{C}_{\geq 2}(\sigma)$, $C\cap(D_\tau\setminus\{a\})\neq\emptyset$; 
\item for every 
$C\in\mathbb{C}_{\geq 2}(\sigma)$ such that $\tau[C\cup\{a\}]$ is primitive, 
$C\subseteq D_\tau\setminus\{a\}$; 
in particular 
$C\subseteq D_\tau\setminus\{a\}$ 
for each $C\in\mathscr{F}$. 
\end{enumerate}
\end{cor}

\subsection{Main results}

Let $\sigma$ be a reversible 2-structure such that $\nu(\sigma)\geq 2$. 
Assume that $\varepsilon(\sigma)=1$, that is, $\sigma$ is complete. 
We consider extensions of $\sigma$ which are identifiable with graphs. 
When $\nu(\sigma)<\aleph_0$, it follows from~\cite[Theorem~2.45]{S71} that there exists a primitive extension $\tau$ of $\sigma$ such that $\varepsilon(\tau)=\varepsilon_s(\tau)=2$ if 
$|V(\tau)\setminus V(\sigma)|\geq\lceil\log_2(\nu(\sigma)+1)\rceil$. 
This result is easily adaptable when $\nu(\sigma)\geq\aleph_0$ by replacing 
$\lceil\log_2(\nu(\sigma)+1)\rceil$ by $\mathfrak{log}_2(\nu(\sigma))$. 

Now assume that $\varepsilon(\sigma)\geq 2$. 
We obtain the following lower bound. 

\begin{prop}\label{P1bound}
For a reversible 2-structure $\sigma$ such that $\varepsilon(\sigma)\geq 2$, we have 
$p(\sigma)\geq\mathfrak{log}_{\varepsilon(\sigma)}(c(\sigma))$.
\end{prop}

\begin{thm}\label{T1bound}
Consider a reversible 2-structure $\sigma$ such that $\varepsilon(\sigma)\geq 2$. If 
$\mathfrak{log}_{\varepsilon(\sigma)}(c(\sigma))\geq\aleph_0$, then 
$p(\sigma)=\mathfrak{log}_{\varepsilon(\sigma)}(c(\sigma))$. 
\end{thm}

\begin{proof}
By Propositon~\ref{P1bound}, it suffices to construct a primitive and faithful extension $\tau$ of $\sigma$ such that 
$|V(\tau)\setminus V(\sigma)|=\mathfrak{log}_{\varepsilon(\sigma)}(c(\sigma))$. 
Let $S'$ be a set such that $S'\cap V(\sigma)=\emptyset$ and 
$|S'|=\mathfrak{log}_{\varepsilon(\sigma)}(c(\sigma))$. 
Thus $\varepsilon(\sigma)^{|S'|}\geq c(\sigma)$. 
We use Proposition~\ref{P1ext} as follows. 
It is easy to construct a primitive and reversible 2-structure $\sigma'$ defined on $S'$ such that 
$\varepsilon(\sigma')=2$ and either $\varepsilon_a(\sigma')=0$ or $\varepsilon_s(\sigma')=0$. 
Hence we can assume that 
$\varepsilon_a(\sigma')\leq\varepsilon_a(\sigma)$ and 
$\varepsilon_s(\sigma')\leq\varepsilon_s(\sigma)$. 
Moreover, since $|S'|=\mathfrak{log}_{\varepsilon(\sigma)}(c(\sigma))\geq\aleph_0$, 
we have $\varepsilon(\sigma)<c(\sigma)$ and hence 
$\varepsilon(\sigma)^{|S'|}>\varepsilon(\sigma)$. 
Finally, consider $C\in\mathscr{C}(\sigma)$. 
We have $|C|\leq c(\sigma)\leq\varepsilon(\sigma)^{|S'|}$.  
By Proposition~\ref{P1ext} applied with $S=C$, 
there exists an extension $\tau_C$ of $\sigma$ and $\sigma'$ to $V(\sigma)\cup V(\sigma')$ such that $\tau_C$ is a faithful extension of $\sigma$ and $\tau_C[C\cup V(\sigma')]$ is primitive. 

Let $a\not\in V(\sigma')$. 
By Lemma~\ref{L2ext}, there exist distinct, primitive and faithful extensions $\sigma''_0$ and $\sigma''_1$ of $\sigma'$ to $V(\sigma')\cup\{a\}$. 
Set $A_0=\overrightarrow{\sigma''_0}(a)$ and $A_1=\overrightarrow{\sigma''_1}(a)$.  
We get $A_0,A_1:S'\longrightarrow E(\sigma')$ and $A_0\neq A_1$. 
As in the proof of Proposition~\ref{P1ext}, we identify the elements of $E(\sigma')$ to elements of $E(\sigma)$. 
In this way, $A_0,A_1\in E(\sigma)^{S'}$. 
By Proposition~\ref{ptraverse}, $\sigma$ admits a traverse $T$. 
Furthermore, $T$ admits a dense bicoloration $\beta$ by Proposition~\ref{dense}. 
Consider the faithful extension $\tau$ of $\sigma$ defined on $V(\sigma)\cup S'$ satisfying
\begin{itemize}
\item for each $C\in\mathscr{C}(\sigma)$, $\tau[C\cup S']=\tau_C[C\cup V(\sigma')]$;
\item for each $v\in V(\sigma)\setminus C(\sigma)$, $\overrightarrow{\tau}(v)_{\restriction S'}=A_{\beta(v)}$. 
\end{itemize} 
For each $C\in\mathscr{C}(\sigma)$, $\tau[C\cup S']=\tau_C[C\cup V(\sigma')]$ is primitive. 
Moreover, consider $v\in V(\sigma)\setminus C(\sigma)$. 
Since $\overrightarrow{\tau}(v)_{\restriction S'}=A_{\beta(v)}=\overrightarrow{\sigma''_{\beta(v)}}(a)$, 
$\tau[S'\cup\{v\}]$ is primitive. 
Consequently 
$\tau[X\cup S']$ is primitive for each $X\in\mathscr{C}(\sigma)\cup\{\{v\}:v\in V(\sigma)\setminus C(\sigma)\}$.

We prove that $\tau$ is primitive. 
Let $D_\tau\in\mathbb{C}_{\geq 2}(\tau)$. 
As $\tau[C\cup S']$ is primitive for each $C\in\mathscr{C}(\sigma)$, \eqref{E1L0bound} 
holds (see Lemma~\ref{L0bound}). 
By applying Lemma~\ref{L0bound} with $\mathscr{F}=\emptyset$, we obtain $D_\tau\cap S'\neq\emptyset$. 
Since $\tau[S']=\sigma'$ is primitive, we have either $|D_\tau\cap S'|=1$ or $D_\tau\cap S'=S'$. 

Suppose for a contradiction that the first instance holds and denote by $s'$ the unique element of $D_\tau\cap S'$. 
Since $|D_\tau|\geq 2$, there is $X\in\mathscr{C}(\sigma)\cup\{\{v\}:v\in V(\sigma)\setminus C(\sigma)\}$ such that 
$D_\tau\cap X\neq\emptyset$. 
As $\tau[X\cup S']$ is imprimitive, 
it would follow that $D_\tau\supseteq X\cup S'$. 

Consequently, we have $D_\tau\cap S'=S'$. 
Let $X\in\mathscr{C}(\sigma)\cup\{\{v\}:v\in V(\sigma)\setminus C(\sigma)\}$. 
Since $\tau[X\cup S']$ is primitive, we obtain 
$D_\tau\cap (X\cup S')=(X\cup S')$. 
Therefore $X\subseteq D_\tau$ for every $X\in\mathscr{C}(\sigma)\cup\{\{v\}:v\in V(\sigma)\setminus C(\sigma)\}$. 
It follows that $D_\tau=V(\sigma)\cup S'$. 
\end{proof}

\begin{thm}\label{T2bound}
Let $\sigma$ be a reversible 2-structure. 
\begin{enumerate} 
\item Assume that $c(\sigma)\geq 2$. 
If $c(\sigma)<\varepsilon(\sigma)$ or if $c(\sigma)=\varepsilon(\sigma)\sigma)\geq\aleph_0$, then 
$p(\sigma)=\mathfrak{log}_{\varepsilon(\sigma)}(c(\sigma))=1$. 
\item Assume that $c(\sigma)=1$. 
If $\varepsilon(\sigma)\geq 3$ or if $\varepsilon(\sigma)=\varepsilon_s(\sigma)=2$, then 
$p(\sigma)\leq 1$. 
\end{enumerate}
\end{thm}

\begin{proof}
Let $a\not\in V(\sigma)$. 
We construct a primitive and faithful extension of $\sigma$ to $V(\sigma)\cup\{a\}$. 

First, assume that $\mathscr{C}(\sigma)\neq\emptyset$. 
Given $C\in\mathscr{C}(\sigma)$, 
$C$ is $e$-complete where $e\in E_s(\sigma)$. 
As $\mathscr{C}(\sigma)\neq\emptyset$, $c(\sigma)\geq 2$. 
Thus $\varepsilon(\sigma)> c(\sigma)$ or $\varepsilon(\sigma)\geq c(\sigma)$ and 
$\varepsilon(\sigma)\geq\aleph_0$, 
Since $c(\sigma)\geq |C|$, 
there exists an injection 
$A:C\longrightarrow E(\sigma)\setminus\{e\}$. 
We consider a faithful extension $\tau_C$ of $\sigma$ defined on $V(\sigma)\cup\{a\}$ satisfying 
$\overrightarrow{\tau_C}(a)_{\restriction C}=A$. 
It is easy to verify that $\tau_C[C\cup\{a\}]$ is primitive. 

Second, assume that $\mathscr{L}(\sigma)\neq\emptyset$ and consider 
$L\in\mathscr{L}(\sigma)$. 
We do not have $\varepsilon(\sigma)=\varepsilon_s(\sigma)=2$. 
Therefore $\varepsilon(\sigma)\geq 3$. 
By Corollary~\ref{P1trav}, there exists a faithful extension $\tau_L$ of $\sigma$ defined on $V(\sigma)\cup\{a\}$ such that $\tau_L[L\cup\{a\}]$ is primitive. 

Third, assume that $\mathscr{P}(\sigma)\neq\emptyset$ and consider 
$P\in\mathscr{P}(\sigma)$. 
By Lemma~\ref{L2ext}, there exists a faithful extension $\tau_P$ of $\sigma$ defined on $V(\sigma)\cup\{a\}$ such that 
$\tau_P[P\cup\{a\}]$ is primitive. 

By Proposition~\ref{ptraverse}, $\sigma$ admits a traverse $T$. 
Let $e_0\in E(\sigma)$. 
As $\varepsilon(\sigma)\geq 3$ or $\varepsilon(\sigma)=\varepsilon_s(\sigma)=2$, consider 
$e_1\in E(\sigma)\setminus\{e_0,(e_0)^\star\}$. 
With each dense bicoloration $\beta$ of $T$, we associate the faithful extension $\tau_\beta$ of $\sigma$ defined on $V(\sigma)\cup\{a\}$ satisfying 
\begin{itemize}
\item for every $C\in\mathscr{C}(\sigma)\cup\mathscr{L}(\sigma)\cup\mathscr{P}(\sigma)$, 
$\tau_\beta[C\cup\{a\}]=\tau_C[C\cup\{a\}]$;
\item for each $v\in V(\sigma)\setminus(C(\sigma)\cup L(\sigma)\cup P(\sigma))$, 
$(v,a)_{\tau_\beta}=e_{\beta(v)}$. 
\end{itemize}

We establish the following for each dense bicoloration $\beta$ of $T$. 
If $\tau_\beta$ is imprimitive, then for every nontrivial clan $C_\beta$ of $\tau_\beta$, we have $a\in C_\beta$ and $C_\beta\setminus\{a\}$ is an inclusive clan of $\sigma$. 
Assume that $\tau_\beta$ is imprimitive and consider a 
nontrivial clan $C_\beta$ of $\tau_\beta$. 

Given $C\in\mathscr{C}(\sigma)$, we have 
$\tau_\beta[C\cup\{a\}]=\tau_C[C\cup\{a\}]$ is primitive. 
Therefore, for every $D\in\mathbb{C}_{\geq 2}(\sigma[C])$, $D$ is not a clan of 
$\tau_\beta[C\cup\{a\}]$ and hence of $\tau_\beta$. 
Thus 
$\mathbb{C}_{\geq 2}(\sigma[C])\cap\mathbb{C}(\tau)=\emptyset$. 
By Corollary~\ref{C0bound} applied with $\mathscr{F}=\mathscr{L}(\sigma)\cup\mathscr{P}(\sigma)$, 
we obtain that $\mathbb{C}_{\geq 2}(\sigma)\cap\mathbb{C}(\tau_\beta)=\emptyset$. 
It follows that $a\in C_\beta$. 

Now we show that $C_\beta\setminus\{a\}\in\mathbb{I}(\sigma)$. 
By Corollary~\ref{C0bound}, 
$X\cap (C_\beta\setminus\{a\})\neq\emptyset$ for every 
$X\in\mathbb{P}_{\geq 2}(\sigma)$. 
Furthermore $C(\sigma)\cup L(\sigma)\cup P(\sigma)\subseteq C_\beta\setminus\{a\}$ 
by Corollary~\ref{C0bound}. 
Therefore 
$C_\beta\setminus\{a\}\in\mathbb{I}(\sigma)$. 

We conclude as follows. 
By Proposition~\ref{dense}, $T$ admits a dense bicoloration $\beta$. 
Assume that $\tau_\beta$ is imprimitive and consider a nontrivial clan $C_\beta$ of 
$\tau_\beta$. 
We have $a\in C_\beta$ and $C_\beta\setminus\{a\}\in\mathbb{I}(\sigma)$. 
Since $C(\sigma)\cup L(\sigma)\cup P(\sigma)\subseteq C_\beta\setminus\{a\}$, $(v,a)_{\tau_\beta}=e_{\beta(v)}$ for every $v\in V(\sigma)\setminus(C_\beta\setminus\{a\})$. 
Clearly $1-\beta$ is also a dense bicoloration of $T$. 
Similarly, assume that $\tau_{1-\beta}$ is imprimitive and consider a nontrivial clan 
$C_{1-\beta}$ of $\tau_{1-\beta}$. 
We have $a\in C_{1-\beta}$, $C_{1-\beta}\setminus\{a\}\in\mathbb{I}(\sigma)$ and 
$(v,a)_{\tau_{1-\beta}}=e_{1-\beta(v)}$ for every 
$v\in V(\sigma)\setminus(C_{1-\beta}\setminus\{a\})$. 
It follows that 
$(C_{\beta}\setminus\{a\})\cup(C_{1-\beta}\setminus\{a\})=V(\sigma)$. 
By Theorem~\ref{T1inc}, 
$\widetilde{(C_{\beta}\setminus\{a\})}=\widetilde{(C_{1-\beta}\setminus\{a\})}$. 
Hence $\widetilde{(C_{\beta}\setminus\{a\})}=V(\sigma)$. 
Furthermore, by Theorem~\ref{T1inc}, we have:
\begin{itemize}
\item $V(\sigma)\in\mathbb{P}(\sigma)\setminus\mathbb{L}(\sigma)$ and $\lambda_\sigma(V(\sigma))\in E_a(\sigma)$;
\item the linear order $O_{V(\sigma)}$ admits a smallest element $\{u\}$ and a largest element 
$\{v\}$ where $\{u\},\{v\}\in\mathbb{G}_{1}(\sigma[V(\sigma)])$;
\item by interchanging $C_{\beta}\setminus\{a\}$ and $C_{1-\beta}\setminus\{a\}$, we have 
$C_{\beta}\setminus\{a\}=V(\sigma)\setminus\{v\}$ and 
$C_{1-\beta}\setminus\{a\}=V(\sigma)\setminus\{u\}$. 
\end{itemize}
As $C_{\beta}\setminus\{a\}=V(\sigma)\setminus\{v\}$, we get 
$(u,v)_{\sigma}=(u,v)_{\tau_\beta}=(a,v)_{\tau_\beta}=((v,a)_{\tau_\beta})^\star=(e_{\beta(v)})^\star$. 
Similarly, as $C_{1-\beta}\setminus\{a\}=V(\sigma)\setminus\{u\}$, we get 
$(u,v)_{\sigma}=(u,v)_{\tau_{1-\beta}}=(u,a)_{\tau_{1-\beta}}=e_{1-\beta(u)}$. 
Since $\{u\}=\min O_{V(\sigma)}$ and $\{v\}=\max O_{V(\sigma)}$, we have 
$(u,v)_\sigma=\lambda_\sigma(V(\sigma))$. 
It would follow that $e_{\beta(v)}=(\lambda_\sigma(V(\sigma)))^\star$ and 
$e_{1-\beta(u)}=\lambda_\sigma(V(\sigma))$ so that 
$\{e_0,e_1\}=\{\lambda_\sigma(V(\sigma)),(\lambda_\sigma(V(\sigma)))^\star\}$ 
which contradicts  $e_1\in E(\sigma)\setminus\{e_0,(e_0)^\star\}$. 
Consequently $\tau_\beta$ or $\tau_{1-\beta}$ is primitive. 
\end{proof}

The following generalizes \cite[Theorem~1.4]{BI10} for finite graphs. 
In spite of the length of its proof, it is simpler and shorter than the original proof for finite graphs. 

\begin{thm}\label{T3bound}
Given a reversible 2-structure $\sigma$, if 
$2\leq\varepsilon(\sigma)\leq c(\sigma)<\aleph_0$, then 
$p(\sigma)\leq\lceil\log_{\varepsilon(\sigma)}(c(\sigma)+1)\rceil$.
\end{thm}

\begin{proof}
Let $C_{{\rm max}}\in\mathscr{C}(\sigma)$ such that $|C_{{\rm max}}|=c(\sigma)$. 
Denote by $e_{{\rm max}}$ the element of $E_s(\sigma)$ such that $C_{{\rm max}}$ is $e_{{\rm max}}$-complete. 
Also consider a set $S'$ such that $S'\cap V(\sigma)=\emptyset$ and $|S'|=\lceil\log_{\varepsilon(\sigma)}(c(\sigma)+1)\rceil$. 

By Lemma~\ref{L1ext}, there exists a faithful extension $\tau_{{\rm max}}$ of $\sigma$ defined on $V(\sigma)\cup S'$ satisfying
\begin{itemize}
\item $\tau_{{\rm max}}[C_{{\rm max}}\cup S']$ is primitive; 
\item for every $s'\in S'$, there is $s\in C_{{\rm max}}$ such that 
\begin{equation}\label{E1T3bound}
(\overrightarrow{\tau_{{\rm max}}}(s))^{-1}(\{e_{{\rm max}}\})\cap S'=\{s'\}.
\end{equation}
\end{itemize}

Now consider $C\in\mathscr{C}(\sigma)\setminus\{C_{{\rm max}}\}$. 
Denote by $e_C$ the element of $E_s(\sigma)$ such that $C$ is $e_C$-complete. 
As $|C|\leq c(\sigma)<\varepsilon(\sigma)^{|S'|}$, there is  
\begin{equation}\label{E1bisT3bound}
\text{an injection}\ \  B_C:C\longrightarrow E(\sigma)^{|S'|}\setminus\{(\overline{e_C})_{\restriction S'}\}.
 \end{equation}
There is a faithful extension $\tau_C$ of $\sigma$ defined on $V(\sigma)\cup S'$ such that 
$(\overrightarrow{\tau_C}(v))_{\restriction S'}=B_C(v)\ \text{for each $v\in C$}$. 
We construct a primitive and faithful extension of $\sigma$ to $V(\sigma)\cup S'$ as follows. 
By Proposition~\ref{ptraverse}, $\sigma$ admits a traverse $T$. 
Furthermore $T$ admits a dense bicoloration $\beta$ by Proposition~\ref{dense}. 
There are 
$A_0\neq A_1\in E(\sigma)^{|S'|}\setminus\{(\overline{e})_{\restriction S'}:e\in E(\sigma)\}$. 
We consider the faithful extension $\tau$ of $\sigma$ defined on $V(\sigma)\cup S'$ satisfying
\begin{itemize}
\item $\tau[C_{{\rm max}}\cup S']=\tau_{{\rm max}}[C_{{\rm max}}\cup S']$;
\item for each $C\in\mathscr{C}(\sigma)\setminus\{C_{{\rm max}}\}$, 
$\tau[C\cup S']=\tau_C[C_\cup S']$;
\item for each $v\in V(\sigma)\setminus C(\sigma)$, $(\overrightarrow{\tau}(v))_{\restriction S'}=A_{\beta(v)}$. 
\end{itemize}

To conclude, we prove that $\tau$ is primitive. 
Let $D_\tau\in\mathbb{C}_{\geq 2}(\tau)$. 
It follows from Lemma~\ref{L0bound} applied with $\mathscr{F}=\emptyset$ that 
$D_\tau\cap S'\neq\emptyset$. 

For a contradiction suppose that $C_{{\rm max}}\cap D_\tau=\emptyset$. 
As $\tau[C_{{\rm max}}\cup S']$ is primitive, there is $s'\in S'$ such that $(C_{{\rm max}}\cup S')\cap D_\tau=\{s'\}$. 
Consider $v\in D_\tau\cap V(\sigma)$. 
Since $C_{{\rm max}}$ is a clan of $\sigma$, there is $e\in E(\sigma)$ such that $(v,C_{{\rm max}})_\sigma=e$. 
As $D_\tau$ is a clan of $\tau$, we get $(s',C_{{\rm max}})_\tau=e$. 
It would follow that $e^\star=\overrightarrow{\tau}(x)(s')=\overrightarrow{\tau_{{\rm max}}}(x)(s')$ for every $x\in C_{{\rm max}}$, which contradicts \eqref{E1T3bound}. 
Consequently $C_{{\rm max}}\cap D_\tau\neq\emptyset$. 
Since $S'\cap D_\tau\neq\emptyset$ and since $\tau[C_{{\rm max}}\cup S']$ is primitive, we obtain that 
$C_{{\rm max}}\cup S'\subseteq D_\tau$. 

Let $v\in V(\sigma)\setminus D_\tau$. 
As $C_{{\rm max}}\subseteq D_\tau$, there is $e_v\in E(\sigma)$ such that $(v,C_{{\rm max}})_\sigma=e_v$. 
Since $D_\tau$ is a clan of $\tau$ such that $D_\tau\supseteq C_{{\rm max}}\cup S'$, we get $(v,S')_\tau=e_v$. 
Therefore, for each $v\in V(\sigma)\setminus D_\tau$, there is $e_v\in E(\sigma)$ such that 
\begin{equation}\label{E2T3bound}
(\overrightarrow{\tau}(v))_{\restriction S'}=(\overline{e_v})_{\restriction S'}. 
\end{equation}
Since $A_0\neq A_1\in E(\sigma)^{|S'|}\setminus\{(\overline{e})_{\restriction S'}:e\in E(\sigma)\}$, it follows from 
\eqref{E2T3bound} that $V(\sigma)\setminus C(\sigma)\subseteq D_\tau$. 

Lastly, consider $C\in\mathscr{C}(\sigma)\setminus\{C_{{\rm max}}\}$. 
For a first contradiction suppose that $C\cap D_\tau=\emptyset$. 
There is $e\in E(\sigma)$ such that $(C,C_{{\rm max}})_\sigma=e$. 
Since $D_\tau$ is a clan of $\tau$ such that $D_\tau\supseteq C_{{\rm max}}\cup S'$, we would get $(C,S')_\tau=e$, that is, 
$(\overrightarrow{\tau}(v))_{\restriction S'}=(\overline{e})_{\restriction S'}$ for every $v\in C$, which contradicts \eqref{E1bisT3bound}. 
Thus $C\cap D_\tau\neq\emptyset$. 
Let $u\in C\cap D_\tau$. 
For every $v\in C\setminus D_\tau$, we obtain $v\longleftrightarrow_\tau\{u\}\cup S'$. 
As $C$ is $e_C$-complete, we have $(v,\{u\}\cup S')_\tau=e_C$. 
It would follow from \eqref{E2T3bound} that 
$(\overrightarrow{\tau}(v))_{\restriction S'}=(\overline{e_C})_{\restriction S'}$, that is, $B_C(v)=(\overline{e_C})_{\restriction S'}$ which contradicts \eqref{E1bisT3bound}. 
Therefore $C\subseteq D_\tau$. 

Consequently $C(\sigma)\subseteq D_\tau$ and hence $D_\tau=V(\sigma)\cup S'$. 
\end{proof}

The next follows from Proposition~\ref{P1bound} and Theorem~\ref{T3bound}. 

\begin{cor}\label{C1bound}
For a reversible 2-structure $\sigma$ such that $2\leq\varepsilon(\sigma)\leq c(\sigma)<\aleph_0$, we have 
$\lceil\log_{\varepsilon(\sigma)}(c(\sigma))\rceil\leq p(\sigma)\leq\lceil\log_{\varepsilon(\sigma)}(c(\sigma)+1)\rceil$. 
Consequently, if $c(\sigma)\not\in\{\varepsilon(\sigma)^{k}:k\geq 1\}$, then 
$p(\sigma)=
\lceil\log_{\varepsilon(\sigma)}(c(\sigma))\rceil$. 
\end{cor}

Let $\sigma$ be a reversible 2-structure such that $2\leq\varepsilon(\sigma)<\aleph_0$ and $c(\sigma)=\varepsilon(\sigma)^{k}$ where 
$k\geq 1$. 
By Corollary~\ref{C1bound}, $p(\sigma)=k$ or $k+1$. 
We prove in Theorem~\ref{T3Abound} below that 
\begin{equation}\label{E1C1bound}
p(\sigma)=k+1\ \iff\ \text{there is $e\in E_s(\sigma)$ such that $|\odot_e(\sigma)|=\varepsilon(\sigma)^{k}$.}
\end{equation}
The proofs of the next three results are adapted from that for finite graphs \cite{BI10}. 
We begin by proving \eqref{E1C1bound} from right to left. 

\begin{lem}\label{L1bound}
Let $\sigma$ be a reversible 2-structure such that $2\leq\varepsilon(\sigma)<\aleph_0$ and 
$c(\sigma)=\varepsilon(\sigma)^{k}$ where $k\geq 1$. 
If there is $e\in E_s(\sigma)$ such that $|\odot_e(\sigma)|=\varepsilon(\sigma)^{k}$, then $p(\sigma)=k+1$. 
\end{lem}

\begin{proof}
Let $S'$ be a set such that $S'\cap V(\sigma)=\emptyset$ and 
$|S'|=k$. 
Let $\tau$ be a faithful extension of $\sigma$ defined on $V(\sigma)\cup S'$. 
Consider the function 
\begin{equation*}
\begin{array}{rccl}
A:&\odot_e(\sigma)&\longrightarrow&E(\sigma)^{S'}\\
&v&\longmapsto&\overrightarrow{\tau}(v)_{\restriction S'}.
\end{array}
\end{equation*}
As $|\odot_e(\sigma)|=\varepsilon(\sigma)^{k}$, that is, $|\odot_e(\sigma)|=|E(\sigma)^{S'}|$, we have either $A$ is not injective or 
there exists $v\in\odot_e(\sigma)$ such that $\overrightarrow{\tau}(v)_{\restriction S'}=\bar{e}_{\restriction S'}$. 
In the second instance, 
$V(\tau)\setminus\{v\}$ is a clan of $\tau$. 
In the first, 
there are $v\neq w\in\odot_e(\sigma)$ such that $\overrightarrow{\tau}(v)_{\restriction S'}=\overrightarrow{\tau}(w)_{\restriction S'}$ so that $\{v,w\}$ is a clan of $\tau$. 
In both instances, $\tau$ is imprimitive. 
Therefore $p(\sigma)\neq k$. 
By Corollary~\ref{C1bound}, $p(\sigma)=k+1$. 
\end{proof}

Now we show \eqref{E1C1bound} from from left to right when $\varepsilon(\sigma)=2$ and $k=1$. 

\begin{prop}\label{P2bound}
Let $\sigma$ be a reversible 2-structure such that $c(\sigma)=\varepsilon(\sigma)=2$. 
If $p(\sigma)=2$, then there is $e\in E_s(\sigma)$ such that $|\odot_e(\sigma)|=2$. 
\end{prop}

\begin{proof}
Since $c(\sigma)=\varepsilon(\sigma)=2$, $\sigma$ is associated with a graph. 
We obtain $\mathscr{L}(\sigma)=\emptyset$ and $|C|=2$ for each $C\in\mathscr{C}(\sigma)$. 
Denote the elements of $E(\sigma)=E_s(\sigma)$ by $e_0$ and $e_1$. 
By Proposition~\ref{ptraverse}, $\sigma$ admits a traverse $T$ and 
$T$ admits a dense bicoloration $\beta$ by Proposition~\ref{dense}. 

Consider $a\not\in V(\sigma)$. 
Given $P\in\mathscr{P}(\sigma)$, it follows from Lemma~\ref{L2ext} that there is 
a primitive and faithful extension $\tau_P$ of $\sigma[P]$ defined on $P\cup\{a\}$. 
We associate with $\beta$ a faithful extension $\tau_\beta$ of $\sigma$ to $V(\sigma)\cup\{a\}$ satisfying 
\begin{itemize}
\item for each $C\in\mathscr{C}(\sigma)$, $a\not\longleftrightarrow_{\tau_\beta} C$;
\item for each $P\in\mathscr{P}(\sigma)$, $\tau[P]=\tau_P$;
\item for each $v\in V(\sigma)\setminus(C(\sigma)\cup P(\sigma))$, $(v,a)_{\tau_\beta}=e_{\beta(v)}$. 
\end{itemize}
As $p(\sigma)=2$, $\tau_\beta$ admits a nontrivial clan $D_\beta$. 
It follows from Corollary~\ref{C0bound} applied with $\mathscr{F}=\mathscr{P}(\sigma)$ that 
$a\in D_\beta$, $P(\sigma)\subseteq D_\beta\setminus\{a\}$  and 
$C\cap(D_\beta\setminus\{a\})\neq\emptyset$ for every $C\in\mathscr{C}(\sigma)$. 
Thus $(D_\beta\setminus\{a\})\cup C(\sigma)\in\mathbb{C}(\sigma)$. 
It follows that $(D_\beta\setminus\{a\})\cup C(\sigma)\in\mathbb{I}(\sigma)$. 
Similarly, since $1-\beta$ is also a dense bicoloration of $T$, 
$\tau_{1-\beta}$ admits a nontrivial clan $D_{1-\beta}$ and 
$(D_{1-\beta}\setminus\{a\})\cup C(\sigma)\in\mathbb{I}(\sigma)$. 
By Corollary~\ref{C1inc}, $(D_\beta\setminus\{a\})\cup C(\sigma)$ and $(D_{1-\beta}\setminus\{a\})\cup C(\sigma)$ are comparable under inclusion. 
For instance, assume that $(D_{1-\beta}\setminus\{a\})\cup C(\sigma)\subseteq(D_\beta\setminus\{a\})\cup C(\sigma)$. 
For a contradiction, suppose that there is $v\in V(\sigma)\setminus ((D_\beta\setminus\{a\})\cup C(\sigma))$. 
We get 
$v\not\in C(\sigma)\cup P(\sigma)$ so that 
$(v,a)_{\tau_\beta}=e_{\beta(v)}$ and $(v,a)_{\tau_{1-\beta}}=e_{1-\beta(v)}$. 
It would follow that $(v,D_\beta\setminus\{a\})_\sigma=e_{\beta(v)}$ and 
$(v,D_\beta\setminus\{a\})_\sigma=e_{1-\beta(v)}$. 
Consequently $(D_\beta\setminus\{a\})\cup C(\sigma)=V(\sigma)$. 
As $D_\beta$ is a nontrivial clan of $\tau_\beta$, 
there is $C\in\mathscr{C}(\sigma)$ such that $C\setminus(D_\beta\setminus\{a\})\neq\emptyset$. 
Since $C\cap(D_\beta\setminus\{a\})\neq\emptyset$ and $|C|=2$, there is $c\in C$ such that $C\setminus(D_\beta\setminus\{a\})=\{c\}$. 
Denote by $e$ the element of $E(\sigma)$ such that $C$ is $e$-complete. 
We get $(c,D_\beta)_{\tau_\beta}=e$. 
Finally, consider $v\in V(\sigma)\setminus((D_\beta\setminus\{a\})$ such that $v\neq c$. 
As $C\setminus(D_\beta\setminus\{a\})=\{c\}$, there is $C'\in\mathscr{C}(\sigma)\setminus\{C\}$ such that 
$C'\setminus(D_\beta\setminus\{a\})=\{v\}$. 
Since $(c,D_\beta)_{\tau_\beta}=e$, we obtain $(c,C'\cap(D_\beta\setminus\{a\}))_{\sigma}=e$ and hence 
$(c,v)_{\sigma}=e$ because $C'\in\mathbb{C}(\sigma)$. 
It follows that $c\in\odot_e(\sigma)$. 
As $C\in\mathbb{C}(\sigma)$, we get $C\subseteq\odot_e(\sigma)$. 
In particular $|\odot_e(\sigma)|\geq 2$. 
Since $c(\sigma)=2$ and since $\odot_e(\sigma)\in\mathbb{C}(\sigma)$ such that $\odot_e(\sigma)$ is $e$-complete, 
we obtain $C=\odot_e(\sigma)$. 
\end{proof}

\begin{thm}\label{T3Abound}
For a reversible 2-structure $\sigma$ such that $2\leq\varepsilon(\sigma)<\aleph_0$ and 
$c(\sigma)=\varepsilon(\sigma)^{k}$ where $k\geq 1$, 
$p(\sigma)=k+1$ if and only if there is $e\in E_s(\sigma)$ such that 
$|\odot_e(\sigma)|=\varepsilon(\sigma)^{k}$. 
\end{thm}

\begin{proof}
By Lemma~\ref{L1bound} and Proposition~\ref{P2bound}, it suffices to establish 
\eqref{E1C1bound} from from left to right when $c(\sigma)\geq 3$. 
Hence consider a reversible 2-structure $\sigma$ such that $2\leq\varepsilon(\sigma)<\aleph_0$, 
$c(\sigma)=\varepsilon(\sigma)^{k}$ (where $k\geq 1$), $c(\sigma)\geq 3$ and 
$p(\sigma)=k+1$. 
For convenience set 
$\mathscr{C}_{{\rm max}}(\sigma)=\{C\in\mathscr{C}(\sigma):|C|=c(\sigma)\}$. 
With each $C\in\mathscr{C}_{{\rm max}}(\sigma)$ associate 
$w_C\in C$. 
Set $W=\{w_C:C\in\mathscr{C}_{{\rm max}}(\sigma)\}$. 
We prove that 
\begin{equation}\label{E1T3Abound}
c(\sigma-W)=2^k-1.
\end{equation}
Let $C\in\mathscr{C}_{{\rm max}}(\sigma)$. 
By Theorem~\ref{T1equivalence}, 
the elements of $\mathscr{C}_{{\rm max}}(\sigma)$ are pairwise disjoint. 
Thus 
\begin{equation}\label{E1aT3Abound}
C\setminus W=C\setminus\{w_C\}. 
\end{equation}
Clearly $C\setminus\{w_C\}$ is complete in $\sigma-W$. 
Furthermore $C\setminus\{w_C\}\in\mathbb{C}(\sigma-W)$. 
Therefore $2^k-1=|C\setminus\{w_C\}|\leq c(\sigma-W)$. 

Now we demonstrate that $c(\sigma-W)<c(\sigma)$. 
Consider $C'\in\mathscr{C}(\sigma-W)$ such that $|C'|=c(\sigma-W)$. 
We show that $C'\in\mathbb{C}(\sigma)$. 
We have to verify that 
$w_C\longleftrightarrow_\sigma C'$ for each $C\in\mathscr{C}_{{\rm max}}(\sigma)$. 
Given $C\in\mathscr{C}_{{\rm max}}(\sigma)$, we distinguish the following two cases. 
\begin{itemize}
\item First, assume that there is $v\in(C\setminus\{w_C\})\setminus C'$. 
We have $v\longleftrightarrow_\sigma C'$. 
As $C$ is complete in $\sigma$, $\{v,w_C\}\in\mathbb{C}(\sigma[C])$. 
Thus $\{v,w_C\}\in\mathbb{C}(\sigma)$ and hence 
$w_C\longleftrightarrow_\sigma C'$. 
\item Second, assume that $C\setminus\{w_C\}\subseteq C'$. 
Clearly $w_C\longleftrightarrow_\sigma C'$ when $C\setminus\{w_C\}=C'$. 
Assume that $C'\setminus (C\setminus\{w_C\})\neq\emptyset$. 
Consider $e_{C'}\in E(\sigma)$ such that $C'$ is $e_{C'}$-complete. 
We have $(C\setminus\{w_C\},C'\setminus C)_{\sigma-W}=e_{C'}$. 
Since $C\in\mathbb{C}(\sigma)$, we get $(w_C,C'\setminus C)_\sigma=e_{C'}$. 
Furthermore, 
we have $|C\setminus\{w_C\}|\geq 2$ because $|C|=c(\sigma)$ and $c(\sigma)\geq 3$. 
As $C\setminus\{w_C\}\subseteq C'$, we obtain that $C$ is $e_{C'}$-complete as well. 
Therefore $(w_C,C\setminus\{w_C\})_\sigma=e_{C'}$. 
It follows that $(w_C,C')_\sigma=e_{C'}$. 
\end{itemize}
Consequently $C'\in\mathbb{C}(\sigma)$. 
As $C'$ is complete, it follows from Lemma~\ref{L1clanmax} that there is $D\in\mathscr{C}(\sigma)$ such that $D\supseteq C'$. 
If $D\not\in\mathscr{C}_{{\rm max}}(\sigma)$, then $|C'|\leq|D|<c(\sigma)$. 
If $D\in\mathscr{C}_{{\rm max}}(\sigma)$, then $C'\subseteq D\setminus\{w_D\}$ and hence 
$|C'|<|D|=c(\sigma)$. 
In both instances, we obtain that $|C'|<c(\sigma)$. 
It follows that $c(\sigma-W)<c(\sigma)$ because $|C'|=c(\sigma-W)$. 
Consequently \eqref{E1T3Abound} holds. 

By Corollary~\ref{C1bound}, $p(\sigma-W)=k$. 
Thus there exists a primitive and faithful extension $\tau'$ of $\sigma-W$ such that 
$|V(\tau')\setminus(V(\sigma)\setminus W)|=k$. 
We extend $\tau'$ to $V(\tau')\cup W$ as follows. 
Let $C\in\mathscr{C}_{{\rm max}}(\sigma)$. 
Consider the function 
\begin{equation*}
\begin{array}{rccl}
A_C:&C\setminus\{w_C\}&\longrightarrow&E(\tau')^{V(\tau')\setminus(V(\sigma)\setminus W)}\\
&v&\longmapsto&\overrightarrow{\tau'}(v)_{\restriction V(\tau')\setminus(V(\sigma)\setminus W)}.
\end{array}
\end{equation*}
Since $\tau'$ is primitive, $A_C$ is injective. 
Furthermore, as $|C\setminus\{w_C\}|=\varepsilon(\sigma)^{k}-1$ and 
$|E(\tau')^{V(\tau')\setminus(V(\sigma)\setminus W)}|=\varepsilon(\sigma)^{k}$, there is a unique 
$a_C\in E(\tau')^{V(\tau')\setminus(V(\sigma)\setminus W)}$ such that 
$a_C\neq A_C(v)$ for each $v\in C\setminus\{w_C\}$. 
Consider the faithful extension $\tau$ of $\tau'$ defined on $V(\tau')\cup W$ by 
$\overrightarrow{\tau}(w_C)_{\restriction V(\tau')\setminus(V(\sigma)\setminus W)}=a_C$ for each $C\in\mathscr{C}_{{\rm max}}(\sigma)$. 
As $p(\sigma)=k+1$, $\tau$ is not primitive. 
Consider a nontrivial clan $D_\tau$ of $\tau$. 

Next we show the following. 
Given $C\neq C'\in\mathscr{C}_{{\rm max}}(\sigma)$,
\begin{equation}\label{E1th2}
\text{if $C\cap D_\tau\neq\emptyset$ and $C'\cap D_\tau\neq\emptyset$, 
then $V(\tau')\subseteq D_\tau$.}
\end{equation}
Indeed $D_\tau\cap V(\sigma)\in\mathbb{C}(\sigma)$. 
Since $\widetilde{C},\widetilde{C'}\in\mathbb{P}(\sigma)$ and since 
$(D_\tau\cap V(\sigma))\cap\widetilde{C}\neq\emptyset$ and 
$(D_\tau\cap V(\sigma))\cap\widetilde{C'}\neq\emptyset$, $D_\tau\cap V(\sigma)$ is comparable to 
$\widetilde{C}$ and $\widetilde{C'}$ under inclusion. 
Suppose for a contradiction that $D_\tau\cap V(\sigma)\subsetneq\widetilde{C}$ and 
$D_\tau\cap V(\sigma)\subsetneq\widetilde{C'}$. 
It follows that $C\cap\widetilde{C'}\neq\emptyset$. 
As $\widetilde{C'}\in\mathbb{P}(\sigma)$, $\widetilde{C'}\subsetneq C$ or $C\subseteq\widetilde{C'}$. 
In the first instance, 
$\widetilde{C'}$ would be a nontrivial prime clan of $\sigma[C]$. 
Thus $C\subseteq\widetilde{C'}$ and hence $\widetilde{C}\subseteq\widetilde{C'}$. 
Similarly we get $\widetilde{C'}\subseteq\widetilde{C}$. 
Therefore $\widetilde{C'}=\widetilde{C}$ and it would follow from Proposition~\ref{P1clanmax} that $C=C'$. 
Consequently $\widetilde{C}\subseteq(D_\tau\cap V(\sigma))$ or 
$\widetilde{C'}\subseteq(D_\tau\cap V(\sigma))$. 
For instance, assume that $\widetilde{C}\subseteq(D_\tau\cap V(\sigma))$.  
We get $(D_\tau\cap V(\tau'))\supseteq (\widetilde{C}\setminus W)\supseteq (C\setminus W)$ and 
$C\setminus W=C\setminus\{w_C\}$ by \eqref{E1aT3Abound}. 
Since $\tau'$ is primitive and $|C\setminus\{w_C\}|\geq 2$, we obtain $V(\tau')\subseteq D_\tau$. 
It follows that \eqref{E1th2} holds. 

As $\tau'$ is primitive and $D_\tau\cap V(\tau')\in\mathbb{C}(\tau')$, we have 
either $|D_\tau\cap V(\tau')|\leq 1$ or $D_\tau\supseteq V(\tau')$. 
For a contradiction, suppose that $|D_\tau\cap V(\tau')|\leq 1$. 
Since $D_\tau$ is a nontrivial clan of $\tau$, there is 
$C\in\mathscr{C}_{{\rm max}}(\sigma)$ such that $w_C\in D_\tau$. 
It follows from \eqref{E1th2} that 
\begin{equation}\label{E2th2}
C'\cap D_\tau=\emptyset\ \text{for each 
$C'\in\mathscr{C}_{{\rm max}}(\sigma)\setminus\{C\}$.}
\end{equation} 
Thus $D_\tau\cap W=\{w_C\}$ and there is $v\in V(\tau')$ such that $D_\tau\cap V(\tau')=\{v\}$. 
Clearly $D_\tau=\{v,w_C\}$ and we distinguish the following two cases to obtain a contradiction.
\begin{itemize}
\item Suppose that $v\in V(\sigma-W)$. 
We have $\{v,w_C\}\in\mathbb{C}(\sigma)$. 
As $\{v,w_C\}$ is complete or linear, it follows from Lemma~\ref{L1clanmax} that there is 
$C'\in\mathscr{C}(\sigma)\cup\mathscr{L}(\sigma)$ such that $C'\supseteq\{v,w_C\}$. 
Since $C,C'\in\mathfrak{C}_{\geq 2}(\sigma)$ by Theorem~\ref{T1equivalence}, we get $C=C'$. 
Consequently we would obtain 
$\overrightarrow{\tau}(w_C)_{\restriction V(\tau')\setminus(V(\sigma)\setminus W)}=
\overrightarrow{\tau}(v)_{\restriction V(\tau')\setminus(V(\sigma)\setminus W)}$, that is, 
$A_C(v)=a_C$. 
\item Suppose that $v\in V(\tau')\setminus V(\sigma-W)$. 
Consider $e_C\in E_s(\sigma)$ such that $C$ is $e_C$-complete. 
We have $(w_C,C\setminus\{w_C\})_\sigma=e_C$. 
As $\{v,w_C\}\in\mathbb{C}(\tau)$, we get $(v,C\setminus\{w_C\})_{\tau'}=e_C$. 
Since $A_C$ is injective, the function 
\begin{equation*}
\begin{array}{rcl}
C\setminus\{w_C\}&\longrightarrow&
E(\tau')^{((V(\tau')\setminus(V(\sigma)\setminus W))\setminus\{v\})}\\
u&\longmapsto&
\overrightarrow{\tau'}(u)_{\restriction (V(\tau')\setminus(V(\sigma)\setminus W))\setminus\{v\}}
\end{array}
\end{equation*}
would be injective also and we would have 
$\varepsilon(\sigma)^{k}-1\leq\varepsilon(\sigma)^{k-1}$ which does not hold when 
$\varepsilon(\sigma)^{k}\geq 3$. 
\end{itemize}
Consequently $V(\tau')\subseteq D_\tau$. 
As $D_\tau$ is a nontrivial module of $\tau$, there exists $C\in\mathscr{C}_{{\rm max}}(\sigma)$ such that $w_C\not\in D_\tau$. 
Consider $e\in E_s(\sigma)$ such that $C$ is $e$ complete. 
We have $(w_C,C\setminus\{w_C\})_\sigma=e$ and hence 
$(w_C,V(\tau'))_\tau=e$. 
In particular $(w_C,V(\sigma-W))_\sigma=e$. 
Given $C'\in\mathscr{C}_{{\rm max}}(\sigma)\setminus\{C\}$, we obtain 
$(w_C,C'\setminus\{w_{C'}\})_\sigma=e$. 
Since $C'\in\mathbb{C}(\sigma)$, we get $(w_C,w_{C'})_\sigma=e$. 
It follows that $w_C\in\odot_{e}(\sigma)$. 
As at the end of the proof of Proposition~\ref{P2bound}, we conclude by 
$C=\odot_{e}(\sigma)$. 
\end{proof}

By Theorem~\ref{T2bound}, if $\sigma$ is an asymmetric 2-structure 
such that $\varepsilon(\sigma)\geq 3$, then $p(\sigma)\leq 1$. 
We complete the section with tournaments. 

\begin{thm}\label{T4bound}
Given a 2-structure $\sigma$ such that 
$\varepsilon(\sigma)=\varepsilon_a(\sigma)=2$, we have 
$p(\sigma)\leq 2$. 
Moreover, 
$p(\sigma)=2$ if and only if 
$\sigma$ is a finite linear order such that $\nu(\sigma)$ is odd. 
\end{thm}

\begin{proof}
By Corollary~\ref{C3trav}, it suffices to prove that if $p(\sigma)\geq 2$, then 
$\sigma$ is a linear order. 
Assume that $p(\sigma)\geq 2$. 
Denote by 
$\mathscr{L}_{{\rm odd}}(\sigma)$ the family of $L\in\mathscr{L}(\sigma)$ such that 
$|L|<\aleph_0$ and $|L|$ is odd, and denote by 
$\mathscr{P}_3(\sigma)$ the family of $P\in\mathscr{P}(\sigma)$ such that $|P|=3$. 

Let $a\not\in V(\sigma)$. 
Given $P\in\mathscr{P}(\sigma)\setminus\mathscr{P}_3(\sigma)$, it follows from Lemma~\ref{L2ext} that $\sigma$ admits a primitive and faithful extension $\tau_P$ defined on 
$V(\sigma)\cup\{a\}$ such that $\tau_P[P\cup\{a\}]$ is primitive. 
Given $L\in\mathscr{L}(\sigma)\setminus\mathscr{L}_{{\rm odd}}(\sigma)$, it follows from Corollary~\ref{C3trav} that $\sigma$ admits a primitive and faithful extension $\tau_L$ defined on $V(\sigma)\cup\{a\}$ such that $\tau_L[L\cup\{a\}]$ is primitive. 

By Proposition~\ref{ptraverse}, $\sigma$ admits a traverse $T$. 
Furthermore $T$ admits a dense bicoloration $\beta$ by Proposition~\ref{dense}. 
Set $E(\sigma)=\{e_0,e_1\}$ and 
$\mathscr{F}=(\mathscr{L}(\sigma)\setminus\mathscr{L}_{{\rm odd}}(\sigma))\cup
(\mathscr{P}(\sigma)\setminus\mathscr{P}_3(\sigma))$. 
We consider the faithful extension $\tau$ of $\sigma$ defined on $V(\sigma)\cup\{a\}$ satisfying
\begin{itemize}
\item for each $X\in\mathscr{F}$, 
$\tau[X\cup\{a\}]=\tau_X[C_\cup\{a\}]$;
\item for each $v\in V(\sigma)\setminus(\bigcup\mathscr{F})$, 
$(v,a)_\tau=e_{\beta(v)}$. 
\end{itemize}
Since $p(\sigma)\geq 2$, $\tau$ admits a nontrivial clan $D$. 
By Corollary~\ref{C0bound}, $a\in D$ and $C\cap(D\setminus\{a\})\neq\emptyset$ for every 
$C\in\mathbb{C}_{\geq 2}(\sigma)$. 
Moreover $\bigcup\mathscr{F}\subseteq D\setminus\{a\}$ 
by Corollary~\ref{C0bound}. 
As $p(\sigma)\geq 2$, $D\setminus\{a\}\not\in\mathbb{I}(\sigma)$ by Theorem~\ref{T2inc}. 
It follows that 
$\bigcup(\mathscr{L}_{{\rm odd}}(\sigma)\cup\mathscr{P}_3(\sigma))\not\subseteq D\setminus\{a\}$. 

For a contradiction, suppose that there exists $P\in\mathscr{P}_3(\sigma)$ such that 
$P\not\subseteq D\setminus\{a\}$. 
By Corollary~\ref{C0bound}, $P\cap(D\setminus\{a\})\neq\emptyset$. 
Moreover $P\in\mathbb{P}(\sigma)$because $\sigma[P]$ is primitive. 
Since $P\not\subseteq D\setminus\{a\}$, we get $D\setminus\{a\}\subsetneq P$. 
As $\sigma[P]$ is primitive, $|D\setminus\{a\}|=1$. 
In particular $\mathscr{F}=\emptyset$. 
Furthermore consider 
$C\in(\mathscr{L}_{{\rm odd}}(\sigma)\cup\mathscr{P}_3(\sigma))\setminus\{P\}$. 
By Corollary~\ref{C0bound}, $C\cap(D\setminus\{a\})\neq\emptyset$. 
Therefore $C\cap P\neq\emptyset$ and it would follow from Theorem~\ref{T1equivalence} that $C=P$. 
Consequently $\mathscr{L}(\sigma)\cup\mathscr{P}(\sigma)=\{P\}$. 
By Corollary~\ref{C0bound}, $X\cap(D\setminus\{a\})\neq\emptyset$ for every 
$X\in\mathbb{P}_{\geq 2}(\sigma)$. 
We obtain that $P\in\mathbb{I}(\sigma)$. 
Since $p(\sigma)\geq 2$, it follows from Theorem~\ref{T2inc} that $P=V(\sigma)$. 
Thus $\sigma$ is primitive which would imply that $p(\sigma)=0$. 
It follows that 
$\bigcup\mathscr{P}_3(\sigma)\subseteq D\setminus\{a\}$. 

Consequently there exists $L\in\mathscr{L}_{{\rm odd}}(\sigma)$ such that 
$L\not\subseteq D\setminus\{a\}$. 
By Corollary~\ref{C0bound}, $L\cap(D\setminus\{a\})\neq\emptyset$. 
Clearly $\widetilde{L}\cap(D\setminus\{a\})\neq\emptyset$ and necessarily 
$D\setminus\{a\}\subseteq\widetilde{L}$. 
Lastly, consider $L'\in\mathscr{L}_{{\rm odd}}(\sigma)\setminus\{L\}$. 
By Corollary~\ref{C0bound}, $L'\cap(D\setminus\{a\})\neq\emptyset$ and hence 
$L'\cap\widetilde{L}\neq\emptyset$. 
If $\widetilde{L}\subseteq L'$, then $L\subseteq L'$ and we would get $L=L'$  by 
Theorem~\ref{T1equivalence}. 
Therefore $L'\subseteq\widetilde{L}$. 
It follows that 
$\bigcup\mathscr{L}_{{\rm odd}}(\sigma)\subseteq\widetilde{L}$. 
We obtain that $\widetilde{L}\in\mathbb{I}(\sigma)$. 
As $p(\sigma)\geq 2$, it follows from Theorem~\ref{T2inc} that $\widetilde{L}=V(\sigma)$. 
By Proposition~\ref{P1clanmax}, $V(\sigma)\not\in\mathbb{L}(\sigma)$ and 
$\lambda_\sigma(V(\sigma))\in E_a(\sigma)$. 
Since $p(\sigma)\geq 2$, it follows from Theorem~\ref{T2trav} that 
$\mathbb{G}_{2}(\sigma)=\emptyset$. 
Consequently $\sigma$ is a linear order. 
\end{proof}

\begin{rem}
Let $\sigma$ be an asymmetric 2-structure such that $\varepsilon(\sigma)\geq 3$. 
By Theorem~\ref{T2bound}, $p(\sigma)\leq 1$. 
Assume that $V(\sigma)$ is infinite (or finite and $|V(\sigma)|$ is even). 
We can also obtain $p(\sigma)\leq 1$ by using Theorem~\ref{T4bound} in the following manner. 
A {\em choice class} of $\sigma$ is a subset $c$ of 
$(V(\sigma)\times V(\sigma))\setminus\{(v,v):v\in V(\sigma)\}$ satisfying: for each $e\in E(\sigma)$, 
either $e\subseteq c$ and $e^\star\cap c=\emptyset$ or 
$e^\star\subseteq c$ and $e\cap c=\emptyset$. 
With each choice class $c$ of $\sigma$, associate the asymmetric 2-structure $\sigma_c$ defined on 
$V(\sigma_c)=V(\sigma)$ by $E(\sigma_c)=\{c,c^\star\}.$  
Consider $x\not\in V(\sigma)$. 
By  Theorem~\ref{T4bound}, there exists a primitive and faithful extension $\tau_c$ of $\sigma_c$ defined on $V(\sigma)\cup\{x\}$. 
Let $e\in E(\sigma)$. 
By interchanging $e$ and $e^\star$, assume that 
$e\subseteq c$ and $e^\star\cap c=\emptyset$. 
Consider the 2-structure $\rho_e$ defined on 
$V(\rho_e)=V(\sigma)\cup\{x\}$ by 
$E(\rho_e)=(E(\sigma)\setminus\{e,e^\star\})\cup
\{e\cup((\sigma_c\hookrightarrow\tau_c)(c)\setminus c),e^\star\cup((\sigma_c\hookrightarrow\tau_c)(c^\star)\setminus c^\star)\}$. 
It is easy to verify that $\rho_e$ is a primitive and faithful extension of $\sigma$. 
Thus $p(\sigma)\leq 1$. 
\end{rem}


\begin{thebibliography}{99}

\bibitem{BE65}Z.W. Birnbaum, J.D. Esary, 
Modules of coherent binary systems, 
J. Soc. Indust. Appl. Math. 13 (1965) 444-462. 

\bibitem{BI10}A. Boussa\"{\i}ri, P. Ille, 
Prime bound of a graph, 2011, \verb|http://arxiv.org/abs/1110.2935v1|.

%\bibitem{BHV04}A. Brandst\"{a}dt, C.T. Ho\`ang, J.-M. Vanherpe, 
%On minimal prime extensions of a four-vertex graph in a prime graph, 
%Discrete Math. 288 (2004), 9--17. 

\bibitem{B07}R. Brignall, 
Simplicity in relational structures and its application to permutation classes, Ph.D. Thesis, University of St Andrews, 2007. 

\bibitem{BRV11}R. Brignall, N. Ru\v skuc and V. Vatter, Simple extensions of combinatorial structures, Mathematika (2011) 57, 193--214.

\bibitem{C97}C. Capelle. D\'ecompositions de Graphes et Permutations Factorisantes, Ph.D. Thesis, Universit\'e
Montpellier II, 1997.

\bibitem{CHM02}C. Capelle, M. Habib, F. de Montgolfier, 
Graph decompositions and factorizing permutations, 
Discrete Math. Theor. Comput. Sci. 5 (2002), 55--70. 

\bibitem{CM05} R. McConnell, F. de Montgolfier, 
Linear-time modular decomposition of directed graphs, 
Discrete Appl. Math. 145 (2005), 198--209.

\bibitem{CH93}A. Cournier, M. Habib, An efficient algorithm to recognize prime 
undirected graphs, in: E.W. Mayr (Ed.), Graph-Theoritic Concepts in 
Computer Science, Lecture Notes in Computer Science, Vol. 657, Springer, Berlin, 
1993, 212-224. 

%\bibitem{ER90}A. Ehrenfeucht, G. Rozenberg, 
%Primitivity is hereditary for 2-structures, fundamental study, 
%Theoret. Comput. Sci. 3 (70) (1990), 343--358. 

\bibitem{EHR99}A. Ehrenfeucht, T. Harju, G. Rozenberg, 
The Theory of 2-Structures, A Framework for Decomposition and 
Transformation of Graphs, 
World Scientific, Singapore, 1999. 

%\bibitem{ER90a}A. Ehrenfeucht, G. Rozenberg, 
%Theory of 2-structures, Part I: clans, basic subclasses, and morphisms, fundamental study, 
%Theoret. Comput. Sci. 3 (70) (1990) 343--358. 

\bibitem{EFHM72}P. Erd\H{o}s, E. Fried, A. Hajnal, E.C. Milner, 
Some remarks on simple tournaments, 
Algebra Universalis 2 (1972), 238--245. 

\bibitem{EHM72}P. Erd\H{o}s, A. Hajnal, E.C. Milner, 
Simple one point extensions of tournaments, 
Mathematika 19 (1972), 57--62. 

\bibitem{F53}R. Fra\"{\i}ss\'e, 
On a decomposition of relations which generalizes the sum of ordering relations, 
Bull. Amer. Math. Soc. 59 (1953), 389.

\bibitem{F84}R. Fra\"{\i}ss\'e, L'intervalle en 
th\'eorie des relations, ses g\'en\'eralisations, filtre intervallaire et 
cl\^{o}ture d'une relation, in: M. Pouzet and D. Richard, eds., 
Order, Description and Roles (North-Holland, Amsterdam, 1984) 313--342.

\bibitem{FL71}E. Fried and H. Laskar, Simple tournaments, 
Notices Amer. Math. Soc. 18 (1971) 395.


\bibitem{G67}T. Gallai, 
Transitiv orientierbare Graphen, 
Acta Math. Acad. Sci. Hungar. 18 (1967), 25--66. 


\bibitem{G97}V. Giakoumakis, 
On the closure of graphs under substitution, 
Discrete Math. 177 (1997), 83--97. 

%\bibitem{GO07}V. Giakoumakis, S. Olariu, 
%All minimal prime extensions of hereditary classes of graphs, 
%Theoret. Comput. Sci. 370 (2007), 74--93. 

\bibitem{HPV99} M. Habib, C. Paul, L. Viennot, Partition Refinement techniques: an interesting toolkit, International
Journal of Foundations of Computer Science, 10 (1999), 14--170.

\bibitem{HR94}T. Harju, G. Rozenberg, Decomposition of infinite labeled 2-structures, 
Lecture Notes in Comput. Sci. 812 (1994), 145--158. 

%\bibitem{I91}P. Ille, L'ensemble des intervalles d'une multirelation binaire 
%et reflexive, Z. Math. Logik Grundlag. Math. 37 (1991), 227--256.

\bibitem{I97}P. Ille, 
Indecomposable graphs, 
Discrete Math. 173 (1997), 71--78. 

\bibitem{IW09}P. Ille, R. Woodrow, 
Weakly partitive families on infinite sets, 
Contrib. Discrete Math. 4 (2009), 54--80. 

%\bibitem{I05}P. Ille, La d\'ecomposition intervallaire des structures binaires, 
%La Gazette des Math\'ematiciens 104 (2005), 39--58.

\bibitem{K85}D. Kelly, Comparability graphs, in: I. Rival (Ed.), Graphs and Orders, 
Reidel, Drodrecht, 1985, 3--40. 

\bibitem{MP01}F. Maffray, M. Preissmann, 
``A translation of Tibor Gallai's paper: Transitiv orientierbare 
Graphen,'' 
Perfect Graphs J.L. Ramirez-Alfonsin and B.A.~Reed, (Editors), 
Wiley, New York (2001), pp.~25--66. 

\bibitem{M72}J.W. Moon, Embedding tournaments in simple tournaments, 
Discrete Math. 2 (1972), 389--395. 

%\bibitem{O90} S. Olariu, 
%On the closure of triangle-free graphs under substitution, 
%Inform. Process. Lett. 34 (1990), 97--101. 

%\bibitem{S59}G. Sabidussi, The composition of graphs, 
%Duke Mathematical Journal 26 (1959), 693--696. 

\bibitem{ST93}J.H. Schmerl, W.T. Trotter, 
Critically indecomposable partially ordered sets, graphs, 
tournaments and other binary relational structures, 
Discrete Math. 113 (1993), 191--205. 

\bibitem{S92}J. Spinrad, P4-trees and substitution decomposition, 
Discrete Appl. Math. 39 (1992) 263-291. 

\bibitem{S71}D.P. Sumner, Indecomposable graphs, 
Ph.D. Thesis, University of \mbox{Massachusetts}, 1971. 

\bibitem{S73}D.P. Sumner, Graphs indecomposable with respect to 
the X-join, Discrete Math. 6 (1973) 281-298.

%\bibitem{Z03} I. Zverovich, Extension of hereditary classes with substitutions, 
%Discrete Appl. Math. 128 (2003), 487--509. 

\end{thebibliography}
\end{document}